\documentclass[12pt]{amsart}

\usepackage[
text={440pt,575pt},
headheight=9pt,
centering
]{geometry}

\usepackage{amsmath, amssymb, amsthm, enumerate, bm}
\allowdisplaybreaks 
\usepackage{mathtools}

\usepackage{graphicx}
\usepackage{subcaption}
\usepackage[normalem]{ulem}

\usepackage[all]{xy}

\usepackage{multirow}

\usepackage{pifont}
\usepackage{mathptmx}

\usepackage{xcolor}
\usepackage{hyperref}
\definecolor{darkblue}{RGB}{0,0,160}
\hypersetup{
	colorlinks,%
	citecolor=darkblue,%
	filecolor=black,%
	linkcolor=darkblue,%
	urlcolor=darkblue
}

\usepackage[capitalise]{cleveref}

\usepackage{todonotes}
\makeatletter
\newcommand{\nolisttopbreak}{\vspace{\topsep}\nobreak\@afterheading}
\makeatother

\theoremstyle{definition}
\newtheorem{definition}{Definition}[section]

\newtheorem{theorem}[definition]{Theorem}
\newtheorem{proposition}[definition]{Proposition}
\newtheorem{lemma}[definition]{Lemma}
\newtheorem{corollary}[definition]{Corollary}

\newtheorem*{fact*}{Fact}
\newtheorem{remark}[definition]{Remark}
\newtheorem{example}[definition]{Example}

\newtheorem{problem}[definition]{Problem}

\newcommand{\B}{\mathcal{B}}
\newcommand{\bk}{\mathfrak{b}}

\newcommand{\F}{\mathcal{F}}
\newcommand{\f}{\mathfrak{f}}

\newcommand{\G}{\mathcal{G}}

\newcommand{\Hc}{\mathcal{H}}
\newcommand{\Ik}{\mathfrak{I}}

\renewcommand{\L}{\mathfrak{L}}
\newcommand{\lk}{\mathfrak{l}}
\newcommand{\M}{\mathcal{M}}
\newcommand{\Mk}{\mathfrak{M}}
\newcommand{\NN}{\mathbb{N}}

\newcommand{\R}{\mathcal{R}}

\newcommand{\ZZ}{\mathbb{Z}}

\newcommand\nub{{\boldsymbol 0}}
\newcommand\ab{{\boldsymbol a}}
\newcommand\bb{{\boldsymbol b}}
\newcommand\eb{{\boldsymbol e}}
\newcommand\gb{{\boldsymbol g}}
\newcommand\fb{{\boldsymbol f}}
\newcommand\ib{{\boldsymbol i}}
\newcommand\hb{{\boldsymbol h}}

\newcommand\mb{{\boldsymbol m}}
\newcommand\rb{{\boldsymbol r}}
\newcommand\ub{{\boldsymbol u}}
\newcommand\vb{{\boldsymbol v}}
\newcommand\wb{{\boldsymbol w}}
\newcommand\xb{{\boldsymbol x}}

\newcommand{\restr}[1]{|_{#1}}

\DeclareMathOperator{\id}{id}
\DeclareMathOperator{\dlex}{dlex}
\DeclareMathOperator{\rev}{rev}

\DeclareMathOperator{\lex}{lex}

\DeclareMathOperator{\Sym}{Sym}

\DeclareMathOperator{\Inc}{Inc}

\DeclareMathOperator{\ini}{in}

\DeclareMathOperator{\supp}{supp}

\newcommand\defas{\coloneqq}

\makeatletter
\@namedef{subjclassname@2020}{%
	\textup{2020} Mathematics Subject Classification}
\makeatother

\title[Equivariant lattice bases]{
	Equivariant lattice bases
}

\author{Dinh Van Le}
\address{Universit\"at Osnabr\"uck\\ Osnabr\"uck, Germany}
\email{dlevan@uos.de}

\author{Tim R\"omer}
\address{Universit\"at Osnabr\"uck\\ Osnabr\"uck, Germany}
\email{troemer@uos.de}


\begin{document}
	\begin{abstract}
		We study lattices in free abelian groups of infinite rank that are invariant under the action of the infinite symmetric group, with emphasis on finiteness of their equivariant bases. Our framework provides a new method for proving finiteness results in algebraic statistics. As an illustration, we show that every invariant lattice in $\ZZ^{(\NN\times[c])}$, where $c\in\NN$, has a finite equivariant Graver basis. This result generalizes and strengthens several finiteness results about Markov bases in the literature.
	\end{abstract}
	\maketitle
	
	\section{Introduction}
	
	There are several notions of bases of a lattice in $\ZZ^n$ that are related in a hierarchy as follows (see, e.g. \cite[Chapter 1]{DSS}):
	\begin{align*}
			\text{lattice basis } &\subset   \text{ Markov basis } 
			\subset   \text{ Gr\"{o}bner basis} \\
			&\subset  \text{universal Gr\"{o}bner basis } 
			\subset   \text{ Graver basis.} 
		\end{align*}
	These bases have different origins and applications.	
	Markov bases were introduced in the seminal paper by Diaconis and Sturmfels \cite{DS} as a key tool for sampling algorithms that are used in Fisher's exact test (see \cite[Chapter 9]{Su}). 
	Graver bases \cite{Gr} play an important role in the theory of integer programming as they allow solving linear and various nonlinear integer programming problems in polynomial time (see \cite[Part II]{DHK}).
	As a powerful algebraic tool, Gr\"{o}bner bases can be used to computed Markov bases (via the so-called fundamental theorem of Markov bases \cite[Theorem 3.1]{DS}) and Graver bases (via the Lawrence lifting technique \cite[Chapter 7]{Stu}).
	
	The primary goal of this paper is to extend the above notions of bases to lattices in the free abelian group $\ZZ^{(I)}$, where the basis $I$ need not be finite, and to study them systematically. With motivations from algebraic statistics, we will mainly focus on the case $I=\NN^d\times[c]$, where $\NN$ denotes the set of positive integers, $c,d\in\NN$, and $[c]=\{1,\dots,c\}.$ In this situation, there is a natural action of the infinite symmetric group $\Sym$ on $\ZZ^{(I)}$ that acts diagonally on the unbounded indices and keeps the bounded index unchanged. We are interested in lattices in $\ZZ^{(I)}$ that are invariant under this action and their equivariant bases. More specifically, the main problem that we study in this paper is the following:
	\medskip
	
	\noindent{\bf Problem \ref{main-problem}.}
		Given a $\Sym$-invariant lattice $L$ in $\ZZ^{(I)}$, determine whether $L$ has a finite {equivariant generating set} (respectively, {Markov basis}, {Gr\"{o}bner basis}, {universal Gr\"{o}bner basis}, {Graver basis}).
	\medskip

	This problem is inspired by finiteness results and questions in algebraic statistics involving chains of increasing lattices of finite rank; see, e.g. \cite{AT,AH07,HS12,HoS,SSt}. In fact, any lattice $L$ in $\ZZ^{(I)}$ corresponds to a chain of increasing lattices $\L=(L_n)_{n\ge1}$, where $L_n=L\cap\ZZ^{(I_n)}$ is the truncation of $L$ in the free abelian group $\ZZ^{(I_n)}$ with finite basis $I_n=[n]^d\times [c]$ for $n\ge1$. One may view $L$ as a \emph{global} lattice and its truncations as \emph{local} ones. It is also useful to view $L=\bigcup_{n\ge1}L_n$ as the \emph{limit} of the chain $\L$.  
	
	One of the major questions in algebraic statistics concerns finiteness up to symmetry of Markov bases of certain chains of local lattices $\L=(L_n)_{n\ge1}$. A well-established method for studying this question combines algebraic and combinatorial tools. First, each lattice $L_n$ is assigned a lattice ideal $\Ik_{L_n}$ and the finiteness up to symmetry of Markov bases is translated into the stabilization up to symmetry of the chain of ideals $(\Ik_{L_n})_{n\ge1}$ (via the fundamental theorem of Markov bases mentioned above). Next, the stabilization up to symmetry of the chain $(\Ik_{L_n})_{n\ge1}$ is characterized by the finite generation up to symmetry of the limit ideal $\bigcup_{n\ge1}\Ik_{L_n}$ in an infinite-dimensional polynomial ring. Finally, the latter is usually proved using a combinatorial result called Higman's lemma \cite{Hig}. See, e.g. \cite{AH07,HS12} for details. See also \cite{HM} for a similar approach that is used to prove finiteness up to symmetry of generating sets of chains of local lattices via ideals in Laurent polynomial rings. 
	
	It should be noted that the technique of passing to infinite-dimensional limits has also been applied to study chains of varieties (see \cite{Dr10,Dr14,DK14}), cones and monoids (see \cite{KLR,LR21}), and convex sets (see \cite{LC}). This idea is closely related to a general phenomenon known as \emph{representation stability}, which is formalized in \cite{CF} and has appeared in diverse areas, including linear programming \cite{GHL} and machine learning \cite{LD}.
	
	In this paper, we provide a more direct approach to proving finiteness up to symmetry of lattice bases. Namely, instead of using ring theoretic tools, we characterize finiteness up to symmetry of various bases (generating sets, {Markov bases}, {Gr\"{o}bner bases}, {Graver bases}) of a chain of local lattices in terms of the corresponding finiteness property of the global lattice (see Theorems \ref{stabilization}, \ref{stabilization-others} and \Cref{stabilization-Groebner}). These \emph{local-global} results allow us to study lattices without translating them to ideals, which has advantages in certain situations. 
	
	The main result of this paper, which generalizes and strengthens the finiteness results about Markov bases of hierarchical models in \cite{AT,HoS,SSt}, illustrates the advantages of our approach. 
	\medskip
	
	\noindent{\bf Theorem \ref{Finiteness-Graver}.}
		Suppose that $I=\NN\times[c]$. Then every $\Sym$-invariant lattice in $\ZZ^{(I)}$ has a finite equivariant Graver basis.
	\medskip
	
	It seems difficult and inconvenient to prove the previous result using the ring theoretic tools described above. 
	Our proof employs a close relationship between the Graver basis of a lattice and the Hilbert basis of a related monoid (see \Cref{prop_Hilbert_Graver}). 
	It then uses Higman's lemma to show that this monoid has a finite equivariant Hilbert basis  (see \Cref{lattice-monoid-ext}). Notably, there is a somewhat simpler proof in the case $c=1$ that does not make use of Higman's lemma (see \Cref{thm_Graver}).
	
	As another contribution to the study of \Cref{main-problem}, we show that every $\Sym$-invariant lattice in $\ZZ^{(I)}$ has a finite equivariant generating set (see \Cref{Finiteness-generating}), extending a result of Hillar and Mart\'{i}n del Campo \cite{HM}. We also provide a global version of the independent set theorem of Hillar and Sullivant \cite{HS12}, which is one of the major results on finiteness of equivariant Markov bases (see \Cref{global-independent-set}).
	
	The paper is organized as follows. In \Cref{sec-bases} we define various bases for any lattice in $\ZZ^{(I)}$, discuss their relationship, and prove some general local-global results. \Cref{sec-equi-bases} deals with $\Sym$-invariant lattices and their equivariant bases, focusing on local-global results, which have stronger forms in this setting. The proof and consequences of \Cref{Finiteness-Graver} are given in \Cref{sec-equi-Graver}. Finally, we discuss finiteness of equivariant generating sets and Markov bases in \Cref{sec-Markov}.
	

	\section{Bases of lattices}
	\label{sec-bases}
	
	In this section we extend several notions of lattice bases to lattices of infinite rank. After briefly reviewing their algebraic interpretations, we discuss implications among these bases, as well as a close relationship between the Graver basis of a lattice and the Hilbert basis of its nonnegative monoid. The section closes with local-global results for lattice bases and Hilbert bases.
	
	To begin, let $I$ be an arbitrary set. Denote by $\ZZ^{(I)}$ the free abelian group with basis $I$. Write each element $\ub\in\ZZ^{(I)}$ as $\ub=(u_\ib)_{\ib\in I}$, where $u_\ib\in\ZZ$ and all but finitely many of them are zero. We call
	\[
	\supp(\ub)=\{\ib\in I\mid u_\ib\ne0\}
	\]
	the \emph{support} of $\ub$ and
	\[
	\|\ub\|=\sum_{\ib\in I}|u_\ib|
	\]
	the \emph{norm} of $\ub$.
	Let $\ZZ_{\ge0}^{(I)}$ be the submonoid of $\ZZ^{(I)}$ consisting of elements $\ub=(u_\ib)_{\ib\in I}$ with $u_\ib\ge0$ for all $\ib\in I$. So $\ZZ_{\ge0}^{(I)}$ is nothing but the free commutative monoid generated by $I$. Any $\ub\in\ZZ^{(I)}$ can be decomposed as $\ub=\ub^{+}-\ub^-$ with $\ub^{+},\ub^-\in\ZZ_{\ge0}^{(I)}$ given by
	\[
	u_\ib^+=\max\{u_\ib,0\} 
	\ \text{ and }\ 
	u_\ib^-=\max\{-u_\ib,0\}
	\ \text{ for all }\ \ib\in I.
	\]
	Note that $\ub^{+}$ and $\ub^{-}$ have disjoint supports. A \emph{term order} $\prec$ on $\ZZ_{\ge0}^{(I)}$ is a well-ordering that is additive, i.e. if $\ub\prec\vb$, then $\ub+\wb\prec\vb+\wb$ for all $\wb\in\ZZ_{\ge0}^{(I)}.$
	
	\begin{example}
		\label{orders}
		Later, we will be mostly concerned with the case $I=\NN^d\times[c]$, where $c,d$ are positive integers and $[c]=\{1,\dots,c\}.$ In this case, familiar term orders on $\ZZ_{\ge0}^{n}$ can be extended to $\ZZ_{\ge0}^{(I)}$. Let us consider first the subcase $I=\NN^d$ (i.e. $c=1$) for the sake of clarity. Denote by $B=\{\eb_\ib\}_{\ib\in\NN^d}$ the standard basis of $\ZZ^{(I)}$. The definition below works for any well-ordering on $B$. However, for our purposes, we always employ a particular well-ordering $\prec$ on $B$ that is a modification of the lexicographic order on $\NN^d$. Let $\max(\ib)$ denote the maximal component of $\ib\in\NN^d$. Then for $\eb_\ib,\eb_{\ib'}\in B$ we define $\eb_\ib\prec\eb_{\ib'}$ if either $\max(\ib)< \max(\ib')$, or $\max(\ib)=\max(\ib')$ and the leftmost nonzero component of the vector $\ib-\ib'$ is negative. For instance, when $d=2$, this order reads
		\[
		e_{(1,1)}<e_{(1,2)}<e_{(2,1)}<e_{(2,2)}<e_{(1,3)}<e_{(2,3)}<e_{(3,1)}<\cdots.
		\]
		Now list the terms of each element $\ub=\sum_{\ib\in I}u_\ib \eb_\ib\in\ZZ^{(I)}$ according to the order $\prec$. Let $\f(\ub)$ and $\lk(\ub)$ be respectively the coefficients of the first and last nonzero terms of $\ub$. So if
		\[
		\ub=2e_{(1,2)} + e_{(1,3)}+3e_{(3,1)}-4e_{(2,4)} ,
		\]
		then $\f(\ub)=2$ and $\lk(\ub)=-4$. Let $\vb,\wb\in\ZZ_{\ge0}^{(I)}$. One can check that the following orders are term orders on $\ZZ_{\ge0}^{(I)}$:
		\begin{enumerate}
			\item 
			\emph{Lexicographic order}: $\vb\prec_{\lex} \wb$ if $\lk(\vb-\wb)<0.$
			\item
			\emph{Degree lexicographic order}: $\vb\prec_{\dlex} \wb$ if either $\|\vb\|<\|\wb\|$, or $\|\vb\|=\|\wb\|$ and $\lk(\vb-\wb)<0.$
			\item
			\emph{Reverse lexicographic order}: $\vb\prec_{\rev} \wb$ if either $\|\vb\|<\|\wb\|$, or $\|\vb\|=\|\wb\|$ and $\f(\vb-\wb)>0.$
		\end{enumerate}
		
		These orders are defined in the same manner for the general case when $c\ge1$, provided that a well-ordering on the standard basis of $\NN^d\times[c]$ is given. Such an order can be obtained by extending the order $\prec$ above as follows: for standard basis elements $\eb_{\ib,j}$ and $\eb_{\ib',j'}$ of $\ZZ^{(I)}$ with $\ib,\ib'\in\NN^d$ and $j, j'\in [c]$, we define 
		\[
		\eb_{\ib,j}\prec\eb_{\ib',j'}
		\ \text{ if either }\ j<j',
		\text{ or } j= j'
		\text{ and } \eb_\ib\prec\eb_{\ib'}.
		\]
	\end{example}
	
	In what follows, by a \emph{lattice} $L$ in $\ZZ^{(I)}$ we mean a subgroup of $\ZZ^{(I)}$. Thus, in particular, $L$ is itself a free abelian group. For $\ub\in\ZZ_{\ge0}^{(I)}$ the \emph{$L$-fiber} of $\ub$ is defined as
	\[
	F_L(\ub)=\{\vb\in\ZZ_{\ge0}^{(I)}\mid \ub-\vb\in L\}.
	\]
	We will simply write this fiber as $F(\ub)$ if $L$ is clear from the context.
	For a subset $\B\subseteq L$ let $G(\ub,\B)$ be the (undirected) graph with vertices $F(\ub)$ and edges $(\vb,\wb)$ if $\vb-\wb\in\pm \B.$ 
	In general, the fiber $F(\ub)$ might be infinite, and hence $G(\ub,\B)$ need not be a finite graph.
	
	Let $\prec$ be a term order on $\ZZ_{\ge0}^{(I)}$. We denote by $G_\prec(\ub,\B)$ the directed graph with underlying undirected graph $G(\ub,\B)$ and an edge $(\vb,\wb)$ being directed from $\vb$ to $\wb$ if $\wb\prec\vb$. Since $\prec$ is a well-ordering, $F(\ub)$ always has a unique $\prec$-minimal element for any $\ub\in\ZZ_{\ge0}^{(I)}$. 
	
	We need another partial order on $\ZZ^{(I)}$ defined as follows: 
	\[
	\ub\sqsubseteq\vb
	\ \text{ if }\ u_\ib v_\ib\ge0
	\ \text{ and }\ |u_\ib|\le|v_\ib|
	\ \text{ for all }\ \ib\in I.
	\]
	Obviously, if $\ub\sqsubseteq\vb$, then $\|\ub\|\le\|\vb\|$. This implies that any non-empty subset of $\ZZ^{(I)}$ has a minimal element with respect to $\sqsubseteq$. In other words, $\sqsubseteq$ is a well-founded ordering. Note that the restriction of $\sqsubseteq$ on $\ZZ_{\ge0}^{(I)}$ is a partial order that is coarser than any term order $\prec$ on $\ZZ_{\ge0}^{(I)}$, i.e. if  $\ub\sqsubseteq\vb$, then $\ub\prec\vb$ for all $\ub,\vb\in\ZZ_{\ge0}^{(I)}$.
	
	Let us now extend several notions of lattice bases to our setting. In the next definition, the classical notions correspond to the case where $I$ is finite.
	
	\begin{definition}
		\label{def:bases}
		Let $L\subseteq\ZZ^{(I)}$ be a lattice and $\B\subseteq L$. Let $\prec$ be a term order on $\ZZ_{\ge0}^{(I)}$. We say that $\B$ is 
		\begin{enumerate}
			\item 
			a \emph{generating set} of $L$ if $\B$ generates the abelian group $L$, i.e. $L=\ZZ \B$;
			\item
			a \emph{Markov basis} of $L$ if the graph $G(\ub,\B)$ is connected for every $\ub\in\ZZ_{\ge0}^{(I)}$;
			\item
			a \emph{Gr\"{o}bner basis} of $L$ with respect to $\prec$ if for every $\ub\in\ZZ_{\ge0}^{(I)}$ there exists a directed path in $G_\prec(\ub,\B)$ from $\ub$ to the unique $\prec$-minimal element of $F(\ub)$;
			\item
			a \emph{universal Gr\"{o}bner basis} of $L$ if $\B$ is a Gr\"{o}bner basis of $L$ with respect to every term order on $\ZZ_{\ge0}^{(I)}$;
			\item
			the \emph{Graver basis} of $L$ if $\B$ is the set of all $\sqsubseteq$-minimal elements in $L\setminus\{\nub\}$.
		\end{enumerate}
	\end{definition}

	\begin{remark}
		The following reformulations of Markov basis and Graver basis are well-known in the classical case, and continue to hold in the general setting of \Cref{def:bases}.
		\begin{enumerate}
			\item 
			$\B$ is a Markov basis of $L$ if and only if for any $\ub\in\ZZ_{\ge0}^{(I)}$ and any $\vb,\wb\in F(\ub)$ there exist $\bb_1,\dots,\bb_s\in\pm \B$ such that
			\[
			\vb=\wb+\sum_{i=1}^s\bb_i
			\quad\text{and}\quad
			\wb+\sum_{i=1}^t\bb_i\in F(\ub) \ \text{ for all }\ 
			t\in [s].
			\]
			This is immediate from the definition: even if the graph $G(\ub,\B)$ is infinite, the connectivity of the graph still means that any two vertices are connected by a finite path; see, e.g. \cite[Chapter 8]{Di}.
			\item
			The Graver basis of $L$ is precisely the set of all elements $\ub\in L\setminus\{\nub\}$ that do not have a conformal decomposition. Here, a \emph{conformal decomposition} of $\ub$ is an expression of the form $\ub=\vb+\wb$ with $\vb,\wb\in L\setminus\{\nub\}$ and $|u_\ib|=|v_\ib|+|w_\ib|$ for all $\ib\in I$.
			Indeed, if $\ub=\vb+\wb$ is a conformal decomposition, then $\vb,\wb\sqsubset\ub$, and hence $\ub$ is not $\sqsubseteq$-minimal in $L\setminus\{\nub\}$. Conversely, if $\ub$ is not $\sqsubseteq$-minimal in $L\setminus\{\nub\}$, then there exists $\vb\in L\setminus\{\nub\}$ with $\vb\sqsubset\ub$, leading to a conformal decomposition $\ub=\vb+(\ub-\vb)$.
		\end{enumerate}
	\end{remark}
	
	As in the classical case, each of the above notion of basis has an algebraic interpretation, which we now briefly discuss. Let $R=K[x_\ib\mid \ib\in I]$ be the polynomial ring over a field $K$ with variables indexed by $I$. Also, let $R^\pm=R[x_\ib^{-1}\mid \ib\in I]$ be the ring of Laurent polynomials in the variables of $R$. For any ideal $\Ik\subseteq R$ let $\Ik^\pm=\Ik R^\pm$ denote the extension of $\Ik$ in $R^\pm$. So each element of $\Ik^\pm$ is of the form $fg^{-1}$, where $f\in \Ik$ and $g$ is a monomial in $R$. Any element $\ub=(u_\ib)_{\ib\in I}\in\ZZ_{\ge0}^{(I)}$ corresponds to a monomial $\xb^{\ub}=\prod_{\ib\in I}x_\ib^{u_\ib}\in R.$ Likewise, any term order $\prec$ on $\ZZ_{\ge0}^{(I)}$ corresponds to a term order on $R$, which is also denoted by $\prec$.
	Let $L$ be a lattice in $\ZZ^{(I)}$. The binomial ideal
	\[
	\mathfrak{I}_L=\langle \xb^{\ub}-\xb^\vb
	\mid\ub,\vb\in\ZZ_{\ge0}^{(I)},\ub-\vb\in L\rangle
	\subseteq R
	\]
	is called the \emph{lattice ideal} associated to $L$.
	It is easy to see that
	\[
	\Ik_L=\langle \xb^{\ub^+}-\xb^{\ub^-}
	\mid\ub\in L\rangle.
	\]
	A binomial $\xb^{\ub^+}-\xb^{\ub^-}\in \Ik_L$ is said to be \emph{primitive} if there is no other binomial $\xb^{\vb^+}-\xb^{\vb^-}$ in $\Ik_L$ such that $\xb^{\vb^+}$ divides $\xb^{\ub^+}$ and $\xb^{\vb^-}$ divides $\xb^{\ub^-}$. One can check that $\xb^{\ub^+}-\xb^{\ub^-}$ is primitive if and only if $\ub$ is a $\sqsubseteq$-minimal element of $L\setminus\{\nub\}$.
	For a subset $\B\subseteq L$ denote $\bk(\B)=\{\xb^{\bb^+}-\xb^{\bb^-} \mid\bb\in \B\}$ and
	\[
	\Ik_\B=\langle\bk(\B)\rangle
	=\langle \xb^{\bb^+}-\xb^{\bb^-}
	\mid\bb\in \B\rangle.
	\]
	
	The next result gives algebraic formulations of lattice bases and can be proved in a similar manner to the classical case; see \cite[Lemma 23]{HM}, \cite[Chapter 5]{Stu} and \cite[Chapter 11]{DHK}. The details are left to the reader.
	
	\begin{proposition}
		\label{bases_algebraic}
		Let $L\subseteq\ZZ^{(I)}$ be a lattice and $\B\subseteq L$. The following statements hold:
		\begin{enumerate}
			\item 
			$\B$ is a generating set of $L$ if and only if $\Ik_L^{\pm}=\Ik_\B^{\pm}$.
			\item
			$\B$ is a Markov basis of $L$ if and only if $\Ik_L=\Ik_\B$.
			\item
			$\B$ is a Gr\"{o}bner basis of $L$ with respect to a term order $\prec$ if and only if $\bk(\B)$ is a Gr\"{o}bner basis of the lattice ideal $\Ik_L$ with respect to $\prec$.
			\item
			$\B$ is a universal Gr\"{o}bner basis of $L$ if and only if $\bk(\B)$ is a universal Gr\"{o}bner basis of $\Ik_L$.
			\item
			$\B$ is the Graver basis of $L$ if and only if $\bk(\B)$ is the set of primitive binomials of $\Ik_L$.
		\end{enumerate} 
	\end{proposition}

	Among different bases of a lattice there is the following relationship that is well-known in the classical case; see, e.g. \cite[Section 1.3]{DSS}.
	
	\begin{proposition}
		\label{implications}
		For any lattice $L\subseteq\ZZ^{(I)}$, the following implications hold for its bases:
		\begin{align*}
			\text{Graver basis } &\Rightarrow   \text{ universal Gr\"{o}bner basis } 
			\Rightarrow   \text{ Gr\"{o}bner basis} \\
			&\Rightarrow   \text{ Markov basis } 
			\Rightarrow   \text{ generating set.} 
		\end{align*}
	\end{proposition}
	
	\begin{proof}
		We only prove the first implication, since the last three ones are immediate from \Cref{def:bases} (or \Cref{bases_algebraic}). The argument is similar to the proof of \cite[Lemma 4.6]{Stu}, but we include it here for the convenience of the reader. Let $\G$ be the Graver basis of $L$ and $\prec$ a term order on $\ZZ_{\ge0}^{(I)}$. Denote by $\ini_\prec(\Ik_\G)$ the initial ideal of $\Ik_\G$ with respect to $\prec$. Also, let 
		$L_{\succ}=\{\ub\in L:\ub\succ\nub\}
		=\{\ub\in L:\ub^+\succ\ub^-\}.$  
		It suffices to show that $\xb^{\ub^+}\in \ini_\prec(\Ik_\G)$ for every $\ub\in L_{\succ}$. Suppose on the contrary that $\xb^{\ub^+}\not\in \ini_\prec(\Ik_\G)$ for some $\ub\in L_{\succ}$. Since $\prec$ is a well-ordering, we may choose such $\ub$ with $\ub^-$ being $\prec$-minimal. We show that $\xb^{\ub^-}\not\in \ini_\prec(\Ik_\G).$ Indeed, if $\xb^{\ub^-}\in \ini_\prec(\Ik_\G)$, then there exists $\bb\in \G\cap L_{\succ}$ such that $\xb^{\bb^+}$ divides $\xb^{\ub^-}$, say, $\xb^{\ub^-}=\xb^{\ab}\xb^{\bb^+}$ for some $\ab\in\ZZ_{\ge0}^{(I)}$. From $\ub,\bb\succ\nub$ it follows that $\ub+\bb\succ\nub$, i.e. $\ub+\bb\in L_{\succ}$. Moreover, since
		\[
		(\ub+\bb)^+-(\ub+\bb)^- = \ub+\bb
		=\ub^+-\ub^-+ \bb^+-\bb^-
		=\ub^+-(\ab+\bb^-)
		\]
		and $(\ub+\bb)^+,(\ub+\bb)^-$ have disjoint supports, we get 
		$(\ub+\bb)^+\sqsubseteq\ub^+$ and $(\ub+\bb)^-\sqsubseteq \ab+\bb^-$. Consequently, $\xb^{(\ub+\bb)^+}\not\in \ini_\prec(\Ik_\G)$ and
		\[
		(\ub+\bb)^-\prec \ab+\bb^-
		\prec \ab+\bb^+ = \ub^-.
		\]
		But this contradicts the minimality of $\ub^-$. Hence, we must have $\xb^{\ub^-}\not\in \ini_\prec(\Ik_\G).$
		
		Now let $\vb\in \G$ be a $\sqsubseteq$-minimal element of $L$ with $\vb\sqsubseteq\ub$. Then $\xb^{\vb^+}$ divides $\xb^{\ub^+}$ and $\xb^{\vb^-}$ divides $\xb^{\ub^-}$. Since either $\xb^{\vb^+}$ or $\xb^{\vb^-}$ belongs to $\ini_\prec(\Ik_\G)$, so does either $\xb^{\ub^+}$ or $\xb^{\ub^-}$. This contradiction completes the proof.
	\end{proof}
	
	Graver bases of lattices are also closely related to Hilbert bases of monoids. To present this relationship, let us first recall some notions. Given a monoid $M\subseteq\ZZ^{(I)}$, we say that $M$ is \emph{generated} by a subset $A\subseteq M$ if any element of $M$ is a \emph{$\ZZ_{\geq 0}$-linear combination} of elements of $A$, i.e.
	$$
	M=\Big\{\sum_{i=1}^lm_i\ab_i\mid l\in\NN,\ \ab_i\in A,\
	m_i\in\ZZ_{\geq 0}\Big\}.
	$$
	Any minimal generating set of $M$ is called a \emph{Hilbert basis}.
	It is known that if $M$ is contained in $\ZZ_{\ge0}^{(I)}$ and finitely generated, then it has a unique Hilbert basis that consists of all \emph{irreducible} elements, i.e. those elements $\ub\in M\setminus\{\nub\}$ with the property that if $\ub=\vb+\wb$ for $\vb,\wb\in M$, then either $\vb=\nub$ or $\wb=\nub$; see, e.g. \cite[Definition 2.15]{BG}. This can be extended to our general setting as follows. 
	
	\begin{lemma}
		\label{lem_Hilbert_monoid}
		Let $M\subseteq \ZZ_{\ge0}^{(I)}$ be a monoid. Then the following statements hold:
		\begin{enumerate}
			\item 
			$M$ has a unique Hilbert basis $\Hc$ that consists of all irreducible elements of $M$.
			\item
			If there is a lattice $L\subseteq \ZZ^{(I)}$ such that $M=L\cap \ZZ_{\ge0}^{(I)}$, then $\Hc$ is exactly the set of all $\sqsubseteq$-minimal elements of $M\setminus\{\nub\}$. In particular, $\Hc=\G\cap\ZZ_{\ge0}^{(I)}$, where $\G$ denotes the Graver basis of $L$.
		\end{enumerate}
	\end{lemma}
	
	\begin{proof}
		(i) Let $\Hc$ denote the set of all irreducible elements of $M$. Obviously, $\Hc$ is contained in any generating set of $M$. So to prove $\Hc$ is the unique Hilbert basis of $M$, it suffices to show that $\Hc$ generates $M$, i.e. every element $\ub\in M$ can be written as a $\ZZ_{\geq 0}$-linear combination of elements of $\Hc$. We proceed by induction on $\|\ub\|$. The case $\|\ub\|=0$ is trivial. Suppose $\|\ub\|>0$. If $\ub$ is irreducible, then we are done. Otherwise, there exist $\vb,\wb\in M\setminus\{\nub\}$ such that $\ub=\vb+\wb$. This implies $\|\wb\|,\|\vb\|<\|\ub\|$ since $\vb,\wb\in\ZZ_{\ge0}^{(I)}$. So the induction hypothesis yields the desired conclusion. 
		
		(ii) Evidently, every $\sqsubseteq$-minimal element of $M\setminus\{\nub\}$ is irreducible. Therefore, it remains to show that any irreducible element $\ub\in M$ is also $\sqsubseteq$-minimal in $M\setminus\{\nub\}$. Suppose not. Then there would exist $\vb\in M\setminus\{\nub\}$ with $\vb\sqsubset \ub$. From 
		$\ub-\vb\in L\cap\ZZ_{\ge0}^{(I)}=M$
		and $\ub-\vb\ne\nub,$ 
		it follows that $\ub=\vb+(\ub-\vb)$ is reducible, a contradiction. 
	\end{proof}

	We conclude this section with results of the \emph{local-global} type that was developed in \cite{KLR,LR21}. From now on we restrict our attention to the case $I=\NN^d\times[c]$, where $c$ and $d$ are positive integers. For any $n\in\NN$ let $I_n=[n]^d\times[c]\subset I$. Then each group $\ZZ^{(I_n)}$ can be regarded as a subgroup of $\ZZ^{(I)}$ via the natural inclusion. More specifically, each element $(u_\ib)_{\ib\in I_n}\in\ZZ^{(I_n)}$ is identified with $(u_\ib)_{\ib\in I}\in\ZZ^{(I)}$, where $u_\ib=0$ for all $\ib\in I\setminus I_n.$ In this manner, we obtain an increasing chain of subgroups
	\[
	\ZZ^{(I_1)}\subset \ZZ^{(I_2)}\subset
	\cdots\subset \ZZ^{(I_n)}\subset
	\cdots,
	\]
	whose limit is $\ZZ^{(I)}=\bigcup_{n\ge1}\ZZ^{(I_n)}.$ By intersecting with this chain, any lattice $L\subseteq\ZZ^{(I)}$ can be approximated through a chain of sublattices of finite rank
	\[
	L_1\subset L_2\subset
	\cdots\subset L_n\subset
	\cdots.
	\]
	Here, each $L_n=L\cap\ZZ^{(I_n)}$ is a truncation of $L$, and conversely, $L=\bigcup_{n\ge1}L_n$ is the limit of the chain $(L_n)_{n\ge1}$. So one can investigate $L$ (viewed as a \emph{global} lattice) using its truncations $L_n$ (viewed as \emph{local} lattices), and vice versa. 
	It should be mentioned that similar ideas have been applied to study ideals, varieties, monoids, and cones in infinite dimensional ambient spaces; see, e.g. \cite{AH07,Dr14,HS12,IY,KLR,LR21}.

	The next result shows how to construct bases of a global lattice from its local bases. For Graver bases the other way around is also possible. 
	
	\begin{proposition}
		\label{prop_local_global}
		Let $L\subseteq\ZZ^{(I)}$ be a lattice with truncations $L_n=L\cap\ZZ^{(I_n)}$ for $n\ge1$.
		\begin{enumerate}
			\item 
			If $\B_n\subseteq L_n$ is a generating set (respectively, Markov basis, universal Gr\"{o}bner basis, Graver basis) of $L_n$ for all $n\ge 1$, then $\B=\bigcup_{n\ge1}\B_n$ is so of $L$. This also holds for Gr\"{o}bner bases provided that the term orders considered are compatible: if $\prec$ is a term order on $\ZZ_{\ge0}^{(I)}$ and each $\B_n$ is a Gr\"{o}bner basis of $L_n$ with respect to the restriction of $\prec$ on $\ZZ_{\ge0}^{(I_n)}$, then $\B$ is a Gr\"{o}bner basis of $L$ with respect to $\prec$.
			\item
			If $\G$ is the Graver basis of $L$, then $\G_n=\G\cap\ZZ^{(I_n)}$ is the Graver basis of $L_n$ for all $n\ge1$.
		\end{enumerate}
	\end{proposition}
	
	\begin{proof}
		(i) We prove the statement for Gr\"{o}bner bases. The other cases can be argued similarly and are left to the reader. For a term order $\prec$ on $\ZZ_{\ge0}^{(I)}$ we denote its restriction on each $\ZZ_{\ge0}^{(I_n)}$ also by $\prec$. Assume that $\B_n$ is a Gr\"{o}bner basis of $L_n$ with respect to $\prec$ for all $n\ge1$. Let $\ub\in\ZZ_{\ge0}^{(I)}$ and let $\vb$ be the unique $\prec$-minimal element of the fiber $F_L(\ub)$. Choose $n$ large enough such that $\ub,\vb\in \ZZ_{\ge0}^{(I_n)}$. Then $\vb\in F_{L_n}(\ub)$, and moreover, $\vb$ is the unique $\prec$-minimal element of $F_{L_n}(\ub)$ since $F_{L_n}(\ub)\subseteq F_L(\ub)$. So there exists a directed path in $G_\prec(\ub,\B_n)$ from $\ub$ to $\vb$. As $\B_n\subseteq \B$, this is also a directed path in $G_\prec(\ub,\B)$. 
		
		(ii) If $\ub\sqsubseteq\vb$ and $\vb\in \ZZ^{(I_n)}$, then $\ub\in \ZZ^{(I_n)}$. So if $\G$ is the set of all $\sqsubseteq$-minimal elements in $L\setminus\{\nub\}$, then each $\G_n=\G\cap L_n$ is the set of all $\sqsubseteq$-minimal elements in $L_n\setminus\{\nub\}$.
	\end{proof}
	
	One might ask whether \Cref{prop_local_global}(ii) also holds for other bases of a lattice. The following example shows that this is not the case for generating sets and Markov bases. In the equivariant setting, however, these bases behave better (see \Cref{stabilization,stabilization-others})
	
	\begin{example}
		Consider the case $I=\NN$ (i.e. $c=d=1$). Then as usual, we write $\ZZ^n$ instead of $\ZZ^{(I_n)}$ for $n\ge1$. 
		Denote by $\{\eb_n\}_{n\in\NN}$ the standard basis of $\ZZ^{(\NN)}$. Let
		\[
		\B_1=\emptyset
		\quad\text{and}\quad 
		\B_{n}=\B_{n-1}\cup\{\eb_{n-1}-2\eb_{n},2\eb_{n}\}
		\ \text{ for }\ n\ge 2.
		\]
		Then $\B=\bigcup_{n\ge 1}\B_n$ generates $\ZZ^{(\NN)}$ since 
		\[
		\eb_n=(\eb_n-2\eb_{n+1})+2\eb_{n+1}
		\ \text{ for }\ n\ge 1.
		\]
		Moreover, $\B$ is also a Markov basis of $\ZZ^{(\NN)}$ because
		\[
		\Ik_\B =\langle x_{n-1}-x_{n}^2,\  x_{n}^2-1\mid n\ge 2\rangle
		= \langle x_{n}-1\mid n\ge 1\rangle
		= \Ik_{\ZZ^{(\NN)}}.
		\]
		On the other hand, it is clear that $\B_n=\B\cap\ZZ^n$ and $\eb_n$ does not belong to the lattice generated by $\B_n$ for all $n\ge 1$. 
		Hence, $\B_n$ is not a generating set of $\ZZ^n=\ZZ^{(\NN)}\cap \ZZ^{n}$ for all $n\ge1$. 
	\end{example}
	
	The next result shows that \Cref{prop_local_global}(ii) can be extended to Gr\"{o}bner bases with respect to the lexicographic order $\prec_{\lex}$ defined in \Cref{orders}. It would be interesting to know whether this also holds for other term orders.
	
	\begin{proposition}
		\label{local_global_Groebner}
		Let $L\subseteq\ZZ^{(I)}$ be a lattice with truncations $L_n=L\cap\ZZ^{(I_n)}$ for $n\ge1$ and let $\B\subseteq L$. Then the following statements are equivalent:
		\begin{enumerate}
			\item 
			$\B$ is a Gr\"{o}bner basis of $L$ with respect to $\prec_{\lex}$;
			\item
			$\B_n=\B\cap\ZZ^{(I_n)}$ is a Gr\"{o}bner basis of $L_n$ with respect to $\prec_{\lex}$ for all $n\ge1$. 
		\end{enumerate}
	\end{proposition}
	
	\begin{proof}
		By \Cref{prop_local_global}, it suffices to prove (i) $\Rightarrow$ (ii). For $\ub\in\ZZ_{\ge0}^{(I_n)}$ let $\vb$ be the unique $\prec_{\lex}$-minimal element of $F_L(\ub)$. Then there is a directed path $P$ in $G_{\prec_{\lex}}(\ub,\B)$ from $\ub$ to $\vb$. For any vertex $\wb$ of $P$ it holds that $\wb\prec_{\lex} \ub$, which yields $\wb\in\ZZ^{(I_n)}$ by the definition of $\prec_{\lex}$. Thus, $\vb$ is the unique $\prec_{\lex}$-minimal element of $F_{L_n}(\ub)$ and $P$ is a path in $G_{\prec_{\lex}}(\ub,\B_n)$. Hence, $\B_n$ is a Gr\"{o}bner basis of $L_n$ with respect to $\prec_{\lex}$.
	\end{proof}
	
	Finally, we record the following local-global result for Hilbert bases of monoids. 
	
	\begin{proposition}
		\label{local_global_Hilbert}
		Consider a monoid $M\subseteq \ZZ_{\ge0}^{(I)}$ and its truncated submonoids $M_n=M\cap \ZZ^{(I_n)}$ for $n\ge1$. Let $\Hc$ be a subset of $M$ and denote $\Hc_n=\Hc\cap M_n$ for $n\ge1$. Then the following statements are equivalent:
		\begin{enumerate}
			\item 
			$\Hc$ is the Hilbert basis $M$;
			\item
			$\Hc_n$ is the Hilbert basis $M_n$ for all $n\ge1$.
		\end{enumerate}
	\end{proposition}
	
	\begin{proof}
		For any $n\ge1$, it is clear that $\ub$ is an irreducible of $M_n$ if and only if $\ub\in\ZZ^{(I_n)}$ and $\ub$ is an irreducible of $M$. Therefore, the result follows easily from \Cref{lem_Hilbert_monoid}(i).
	\end{proof}
	
	
	\section{Equivariant lattice bases}
	\label{sec-equi-bases}
	
	We now restrict our attention to the main objects of study in this work, namely, lattices in $\ZZ^{(I)}$ that are invariant under the infinite symmetric group and their equivariant bases. For those lattices and bases, the local-global results presented in the previous section can be substantially improved. 
	Let us begin with symmetric groups and their actions.
	
	\subsection{Symmetric groups}
	\label{sec:Sym}
	
	For $n\in \NN$ let $\Sym(n)$ be the symmetric group on $[n]$.  Identifying $\Sym(n)$ with the stabilizer subgroup of $n+1$ in $\Sym(n+1)$, we obtain an increasing chain of finite symmetric groups
	$$
	\Sym(1)\subset \Sym(2)\subset\cdots\subset \Sym(n)\subset\cdots.
	$$
	The limit of this chain is the infinite symmetric group
	\[
	\Sym(\infty)=\bigcup_{n\geq 1} \Sym(n)
	\]
	that consists of all finite permutations of $\NN$, i.e. permutations that fix all but finitely many elements
	of $\NN$.
	In what follows, for brevity, we will write $\Sym(\infty)$ simply as $\Sym$.
	
	
	\subsection{Sym action}
	\label{sec:Sym-action}
	
	As before, let $I=\NN^d\times[c]$ and $I_n=[n]^d\times[c]$ for $n\in\NN$, where $c$ and $d$ are given positive integers. There is a natural action of $\Sym$ on $\ZZ^{(I)}$, which we now describe. For the sake of clarity, let us first consider the case $c=1$, i.e. $I=\NN^d$. Let $\{\eb_\ib\}_{\ib\in\NN^d}$ denote the standard basis of $\ZZ^{(I)}$. For any $\ib=(i_1,\dots,i_d)\in\NN^d$ and $\sigma\in\Sym$ we define 
	\[
	\sigma(\ib)=(\sigma(i_1),\dots,\sigma(i_d)).
	\]
	Now the action of $\Sym$ on $\ZZ^{(I)}$ is the linear extension of the following action on the basis elements:
	\[
	\sigma(\eb_\ib)=\eb_{\sigma(\ib)}
	\ \text{ for any }  \sigma\in\Sym
	\text{ and } \ib\in\NN^d.
	\]
	So if $\ub=(u_\ib)_{\ib\in\NN^d}=\sum_{\ib\in\NN^d}u_\ib\eb_\ib\in\ZZ^{(I)}$ and $\sigma\in\Sym$, then
	\[
	\sigma(\ub)=\sigma\Big(\sum_{\ib\in\NN^d}u_\ib\eb_\ib\Big)
	=\sum_{\ib\in\NN^d}u_\ib\sigma(\eb_\ib)
	=\sum_{\ib\in\NN^d}u_\ib\eb_{\sigma(\ib)}.
	\]
	When $c>1$, $\Sym$ acts on $\ZZ^{(I)}$ by permuting the \emph{unbounded indices} while keeping the \emph{bounded index} unchanged. More precisely, if $\eb_{\ib,j}$ is a standard basis element of $\ZZ^{(I)}$ with $\ib\in\NN^d$ and $j\in [c]$, then the action of $\Sym$ on $\ZZ^{(I)}$ is defined via
	\[
	\sigma(\eb_{\ib,j})=\eb_{\sigma(\ib),j}
	\ \text{ for any }  \sigma\in\Sym.
	\]
	Observe that this action induces an action of $\Sym(n)$ on $\ZZ^{(I_n)}$ for $n\ge1$.
	
	A subset $A\subseteq\ZZ^{(I)}$ is called \emph{$\Sym$-invariant} if $A$ is stable under the action of $\Sym$, i.e.
	\[
	\Sym(A)\defas\{\sigma(\ub)\mid \sigma\in \Sym,\ \ub\in A\}
	\subseteq A.
	\]
	Similarly, for any $n\ge1$ a subset $A_n\subseteq\ZZ^{(I_n)}$ is \emph{$\Sym(n)$-invariant} if $ \Sym(n)(A_n)\subseteq A_n$. 

	\subsection{Equivariant bases}
	\label{sec:Equivariant-bases}
	
	With motivations from algebraic statistics, one is interested in lattices in $\ZZ^{(I)}$ or $\ZZ^{(I_n)}$ (with $n\ge1$) that are $\Sym$-invariant or $\Sym(n)$-invariant, respectively; see, e.g. \cite{AH07,HM,HS12,KKL}. Such lattices require appropriate notions of bases that reflect the actions of symmetric groups. In \cite{KKL}, equivariant generating sets and equivariant Markov bases for those lattices have been defined. The next definition extends these notions to other bases.
	
	\begin{definition}
	\label{def-equi-bases}
		Let $L\subseteq\ZZ^{(I)}$ (respectively, $L\subseteq\ZZ^{(I_n)}$ with $n\ge1$) be a $\Sym$- (respectively, $\Sym(n)$-)invariant lattice and let $\B\subseteq L$. Then $\B$ is called an \emph{equivariant generating set} (respectively, \emph{equivariant Markov basis}, \emph{equivariant Gr\"{o}bner basis}, \emph{equivariant universal Gr\"{o}bner basis}, \emph{equivariant Graver basis}) of $L$ if $\Sym(\B)$ (respectively, $\Sym(n)(\B)$) is a generating set (respectively, Markov basis, Gr\"{o}bner basis, universal Gr\"{o}bner basis, Graver basis) of $L$.
	\end{definition}
	
	\begin{remark}
	     As pointed out in \cite[Remark 1.1]{KKL}, there is no well-defined notion of an equivariant lattice basis. The reason is that the independence of basis elements is generally not preserved when taking orbits. For instance, among the elements in the $\Sym(2)$-orbit of $(1,-1)\in\ZZ^2$ there is the following nontrivial linear relation:
	     \[
	     (1,-1)+(-1,1)=(0,0).
	     \]
	\end{remark}
	
	Similarly to \Cref{def-equi-bases}, we define:
	
	\begin{definition}
		Let $M\subseteq\ZZ^{(I)}$ (respectively, $M\subseteq\ZZ^{(I_n)}$ with $n\ge1$) be a $\Sym$- (respectively, $\Sym(n)$-)invariant monoid and let $\Hc\subseteq M$. Then $\Hc$ is called an \emph{equivariant Hilbert basis} of $M$ if $\Sym(\Hc)$ (respectively, $\Sym(n)(\Hc)$) is a Hilbert basis of $M$.
	\end{definition}
	
	\begin{example}
		Consider the case $I=\NN^2\times[c]$ with $r\ge1$ (i.e. $d=2$). Let $\eb_{i_1,i_2,j}$ with $i_1,i_2\in\NN$, $j\in[c]$ denote the standard basis elements of $\ZZ^{(I)}$. Then 
		$
		\G=\{\pm \eb_{i_1,i_2,j}\mid i_1,i_2\in\NN, j\in[c]\}
		$ 
		is the Graver basis of $\ZZ^{(I)}$ and
		$
		\Hc=\{\eb_{i_1,i_2,j}\mid i_1,i_2\in\NN, j\in[c]\}
		$ 
		is the Hilbert basis of $\ZZ_{\ge0}^{(I)}$. It follows that
		$
		\G_1=\{\pm \eb_{1,1,j},\ \pm \eb_{1,2,j}\mid j\in[c]\}
		$
		is an equivariant Graver basis of $\ZZ^{(I)}$ and
		$
		\Hc_1=\{\eb_{1,1,j},\ \eb_{1,2,j}\mid j\in[c]\}
		$ 
		is an equivariant Hilbert basis of $\ZZ_{\ge0}^{(I)}$, since $\G=\Sym(\G_1)$ and $\Hc=\Sym(\Hc_1)$. Note that 
		$
		\G_2=\G_1\setminus\{\ub\}\cup\{\sigma(\ub)\}
		$
		and
		$
		\Hc_2=\Hc_1\setminus\{\vb\}\cup\{\sigma(\vb)\}
		$
		for any $\ub\in\G_1,\vb\in\Hc_1$ and $\sigma\in\Sym$ are also an equivariant Graver basis of $\ZZ^{(I)}$ and an equivariant Hilbert basis of $\ZZ_{\ge0}^{(I)}$, respectively. So equivariant Graver bases and equivariant Hilbert bases are not uniquely determined. 
	\end{example}
	
	Concerning equivariant lattice bases, the following problem is of central importance.
	
	\begin{problem}
		\label{main-problem}
		For a $\Sym$-invariant lattice $L\subseteq\ZZ^{(I)}$, decide whether $L$ has a finite {equivariant generating set} (respectively, {Markov basis}, {Gr\"{o}bner basis}, {universal Gr\"{o}bner basis}, {Graver basis}).
	\end{problem}
	
	The remaining part of the paper is devoted to studying this problem. First, let us establish connections between local and global equivariant lattice bases. As seen in the previous section, any lattice in $\ZZ^{(I)}$ is the limit of the chain of its truncations and can therefore be examined using this chain. In what follows, it is convenient to allow more general chains that do not necessarily consist of truncations of a global lattice. By a \emph{chain of lattices} we mean an arbitrary chain $\L=(L_n)_{n\ge1}$ with $L_n$ a lattice in $\ZZ^{(I_n)}$ for all $n\ge1$. This chain is called \emph{$\Sym$-invariant} if the following conditions are satisfied:
	\begin{enumerate}
		\item
		$\L$ is an increasing chain, i.e. $L_m\subseteq L_n$ for all $1\le m\le n$;
		\item
		$L_n$ is $\Sym(n)$-invariant for all $n\ge1.$
	\end{enumerate}
	It is easy to verify that for any $\Sym$-invariant chain $\L=(L_n)_{n\ge1}$, its limit $L=\bigcup_{n\ge1}L_n$ is a $\Sym$-invariant lattice in $\ZZ^{(I)}$. Conversely, if $L\subseteq\ZZ^{(I)}$ is a $\Sym$-invariant lattice, then its truncations $L_n=L\cap\ZZ^{(I_n)}$ form a $\Sym$-invariant chain that is also called the \emph{saturated} chain of $L$. Note that among $\Sym$-invariant chains having the same limit, the saturated chain is the largest one.
	
	Given a $\Sym$-invariant chain of lattices $\L=(L_n)_{n\ge1}$, it holds that
	\[
	\Sym(n)(L_m)\subseteq \Sym(n)(L_n)\subseteq L_n,
	\]
	hence $L_n$ contains the lattice generated by $\Sym(n)(L_m)$ for all $n\ge m\ge 1$. We say that the chain $\L$ \emph{stabilizes} if this containment becomes an equality when $n\ge m\gg0$, i.e. there exists $p\in\NN$ such that
	\[
	L_n=\ZZ\Sym(n)(L_m)
	\ \text{ for all } n\ge m\ge p.
	\]
	This property is interesting because it characterizes the equivariant finite generation of the global lattice, as shown in the next result; see \cite{KLR,LR21} for related results on cones, monoids, and ideals. For $\ub\in\ZZ^{(I)}$ we call the cardinality of $\supp(\ub)$ the \emph{support size} of $\ub$ and denote it by $|\supp(\ub)|$.
	
	\begin{theorem}
		\label{stabilization}
		Let $\L=(L_n)_{n\ge1}$ be a $\Sym$-invariant chain of lattices with limit $L$. Then the following are equivalent:
		\begin{enumerate}
			\item
			$\L$ stabilizes;
			\item
			There exist $p\in\NN$ and an equivariant generating set $\B_p$ of $L_p$ such that $\B_p$ is also an equivariant generating set of $L_n$ for all $n\ge p$;
			\item
			There exist $q, q'\in\NN$ such that for all $n\ge q$ the following hold:
			\begin{enumerate}
				\item
				$ L_{n} = L \cap \ZZ^{(I_n)}$,
				\item
				$L_{n}$ is generated by elements of support size at most $q'$;
			\end{enumerate}
			\item
			$L$ has a finite equivariant generating set.
		\end{enumerate}
	\end{theorem}
	
	The proof of this result requires some preparations. The following lemma is used to reduce the \emph{width} of the support of an equivariant generating set.
	
	\begin{lemma}
		\label{reduce_width}
		Let $S\subseteq I_n=[n]^d\times[c]$ with $|S|=m$. Suppose that $n\ge p=dm+1$. Then there exists $\sigma\in\Sym(n)$ such that $\sigma(S)\subseteq I_p=[p]^d\times[c].$ 
	\end{lemma}
	
	\begin{proof}
		Denote
		\[
		T_S=\bigcup\{\{i_1,\dots,i_d\}\mid (i_1,\dots,i_d,j)\in S
		\ \text{ for some } j\in [c]\}\subseteq [n].
		\]
		We argue by induction on the cardinality $t$ of the set $T_S\setminus[p].$ If $t=0$, then $T_S\subseteq [p]$, and we can simply take $\sigma=\id_{[n]}$, the identity map of $[n]$.
		Now consider the case $t>0$. Choose an element $k\in T_S\setminus[p].$ Since
		\[
		|T_S| \le d|S|\le dm<p,
		\]
		there must exist an $l\in [p]\setminus T_S.$ Let $\pi\in\Sym(n)$ be the transposition that swaps $k$ and $l$. 
		It is easy to see that
		\[
		T_{\pi(S)}\setminus[p]
		=T_S\setminus[p]-\{k\}.
		\]
		So by the induction hypothesis, there exists $\tau\in\Sym(n)$ such that $\tau(\pi(S))\subseteq I_p.$ The proof is completed by taking $\sigma=\tau\circ\pi$.
	\end{proof}
	
	We also need the following technical but useful observation.
	
	\begin{lemma}
		\label{pull_permutation}
		Let $\sigma_1,\dots,\sigma_h\in \Sym$ and let $m,n\in\NN$. Denote 
		\[
		D=\bigcup_{j=1}^h\sigma_j([m])\cap [n].
		\]
		Suppose $n\ge hm+1$. Then there exist $\sigma\in\Sym$ and $\tau_1,\dots,\tau_h\in \Sym(n)$ such that
		\[
		\sigma\restr{D}=\id_D
		\ \text{ and }\ 
		\sigma\circ\sigma_j\restr{[m]}=\tau_j\restr{[m]}
		\ \text{ for all } j\in[h],
		\]
		where $\sigma\restr{D}$ denotes the restriction of $\sigma$ on $D$.
	\end{lemma}
	
	\begin{proof}
		The argument is similar to that of the previous lemma. Let 
		\[
		T(\sigma_1,\dots,\sigma_h)=\bigcup_{j=1}^h\sigma_j([m]).
		\]
		We proceed by induction on $t=|T(\sigma_1,\dots,\sigma_h)\setminus[n]|.$ If $t=0$, then $\sigma_j([m])\subseteq [n]$ for all $j\in[h].$ Consequently, there exist $\tau_1,\dots,\tau_h\in \Sym(n)$ such that $\tau_j\restr{[m]}=\sigma_j\restr{[m]}$ for all $j\in[h]$, and we are done by choosing $\sigma=\id_{\NN}.$
		
		Now assume $t>0$. Choose an element $k\in T(\sigma_1,\dots,\sigma_h)\setminus[n].$ Since
		\[
		|T(\sigma_1,\dots,\sigma_h)|\le \sum_{j=1}^h|\sigma_j([m])|
		=hm< n,
		\]
		there exists $l\in [n]\setminus T(\sigma_1,\dots,\sigma_h).$ Let $\pi\in\Sym$ be the transposition that swaps $k$ and $l$. Put $\sigma_j'=\pi\circ\sigma_j$ for $j\in[h]$.
		Then for any $i\in [m]$ and $j\in[h]$ one has
		\[
		\sigma_j'(i)=\pi(\sigma_j(i))
		=\begin{cases}
			l&\text{if } \sigma_j(i) =k,\\
			\sigma_j(i)&\text{otherwise}.
		\end{cases}
		\]
		It follows that $D':=\bigcup_{j=1}^h\sigma_j'([m])\cap [n]=D\cup\{l\}.$ Moreover, 
		\[
		T(\sigma_1',\dots,\sigma_h')\setminus[n]
		=T(\sigma_1,\dots,\sigma_h)\setminus[n]-\{k\}.
		\]
		So by the induction hypothesis, there exist $\rho\in\Sym$ and $\tau_1,\dots,\tau_h\in \Sym(n)$ such that
		\[
		\rho\restr{D'}=\id_{D'}
		\ \text{ and }\ 
		\rho\circ\sigma_j'\restr{[m]}=\tau_j\restr{[m]}
		\ \text{ for all } j\in[h].
		\]
		Letting $\sigma=\rho\circ\pi$ it is easy to check that $\sigma$ and $\tau_1,\dots,\tau_h$ fulfill the required condition.
	\end{proof}
	
	We are now ready to prove \Cref{stabilization}.
	
	\begin{proof}[Proof of \Cref{stabilization}]
		We will show (i) $\Rightarrow$ (ii) $\Rightarrow$ (iv) $\Rightarrow$ (iii) $\Rightarrow$ (i).
		
		(i) $\Rightarrow$ (ii): When $\L$ stabilizes, we can choose $p\in\NN$ so that $L_n$ is generated by $\Sym(n)(L_p)$ for all $n\ge p$. Thus if $\B_p$ is any equivariant generating set of $L_p$, then $\Sym(n)(\Sym(p)(\B_p))$ generates $L_n$ for all $n\ge p$. Since $\Sym(n)(\Sym(p)(\B_p))=\Sym(n)(\B_p)$, we are done.
		
		(ii) $\Rightarrow$ (iv): Choose $p$ and $\B_p$ as in (ii). We may assume that $\B_p$ is finite because $L_p$ is finitely generated. Since $L=\bigcup_{n\ge1}L_n$, it is easily seen that $\B_p$ is an equivariant generating set of $L$.
		
		(iv) $\Rightarrow$ (iii): Suppose $\B=\{\bb_1,\dots,\bb_h\}$ is a finite equivariant generating set of $L$. Since $L=\bigcup_{n\ge1}L_n$, there exists $m\in\NN$ such that $\B\subseteq L_m\subseteq\ZZ^{(I_m)}$. Thus, each element of $\B$ has support size at most $|I_m|=m^dc$. The $\Sym(n)$-invariance of $L_n$ yields
		\[
		\ZZ\Sym(n)(\B)\subseteq L_n \subseteq L \cap \ZZ^{(I_n)}
		\ \text{ for all } n\ge m.
		\]
		Observe that the action of $\Sym(n)$ does not alter the support size of any element of $\ZZ^{(I_n)}$. So it suffices to prove the existence of some $q\ge m$ such that
		\[
		L \cap \ZZ^{(I_n)}=\ZZ\Sym(n)(\B)
		\ \text{ for all } n\ge q.
		\]
		We show that $q=hm+1$ is such a number. Indeed, let $n\ge q$ and take any $\ub\in L \cap \ZZ^{(I_n)}$. Since $L =\ZZ\Sym(\B)$, there exist $z_1,\dots,z_h\in \ZZ$ and $\sigma_1,\dots,\sigma_h\in \Sym$ such that
		$$
		\ub=\sum_{k=1}^hz_k\sigma_k(\bb_k).
		$$
		Using \Cref{pull_permutation} we find $\sigma\in\Sym$ and $\tau_1,\dots,\tau_h\in \Sym(n)$ satisfying
		\[
		\sigma\restr{D}=\id_D
		\ \text{ and }\ 
		\sigma\circ\sigma_k\restr{[m]}=\tau_k\restr{[m]}
		\ \text{ for all } k\in[h],
		\]
		where $D=\bigcup_{k=1}^h\sigma_k([m])\cap [n]$. Evidently, if $\ib=(i_1,\dots,i_d,j)\in \supp(\ub)$, then $i_l\in D$ for all $l\in[d].$ So from $\sigma\restr{D}=\id_D$ it follows that $\sigma(\ub)=\ub$. Similarly, from $\sigma\circ\sigma_k\restr{[m]}=\tau_k\restr{[m]}$ we get $\sigma\circ\sigma_k(\bb_k)=\tau_k(\bb_k)$ for all $k\in[h].$ Hence
		\[
		\ub=\sigma(\ub)=\sum_{k=1}^hz_k\sigma\circ\sigma_k(\bb_k)
		=\sum_{k=1}^hz_k\tau_k(\bb_k),
		\]
		and therefore, $\ub\in \ZZ\Sym(n)(\B)$.
		
		(iii) $\Rightarrow$ (i): Denote $p=\max\{q,dq'+1\}$ and let $n\ge p$. Then $L_n$ has a generating set $\B_n$ that consists of elements of support size at most $q'$. Take any $\bb\in\B_n$ and set $S=\supp(\bb)$. Since $|S|\le q'$, there exists $\sigma\in\Sym(n)$ such that $\sigma(S)\subseteq I_p$ by virtue of \Cref{reduce_width}. This implies $\sigma(\bb)\in \ZZ^{(I_p)}\cap L=L_p.$ In other words, $\bb=\sigma^{-1}(\bb')$ for some $\bb'\in L_p.$ Hence, we can find a subset $\B_n'\subseteq L_p$ such that $\B_n\subseteq \Sym(n)(\B_n')$. From
		\[
		\ZZ\B_n\subseteq \ZZ\Sym(n)(\B_n')
		\subseteq \ZZ\Sym(n)(L_p) \subseteq L_n
		\]
		and $L_n=\ZZ\B_n$, it follows that $L_n=\ZZ\Sym(n)(L_p)$ for all $n\ge p$. That is, $\L$ stabilizes.
	\end{proof}
	
	As we will see in \Cref{Finiteness-generating}, every $\Sym$-invariant lattice in $\ZZ^{(I)}$ is equivariantly finitely generated. Hence, the statements (i), (ii), (iii) in \Cref{stabilization} always hold true.
	
	Let us now discuss the stabilization of other equivariant lattice bases. The next definition is motivated by \Cref{stabilization}.
	
	\begin{definition}
		Let $\L=(L_n)_{n\ge1}$ be a $\Sym$-invariant chain of lattices. We say that $\L$ \emph{Markov-} (respectively, \emph{Gr\"{o}bner-}, \emph{universal Gr\"{o}bner-}, \emph{Graver-})\emph{stabilizes} if there exist $p\in\NN$ and an equivariant Markov (respectively, Gr\"{o}bner, universal Gr\"{o}bner, Graver) basis $\B_p$ of $L_p$ such that $\B_p$ is also an equivariant Markov (respectively, Gr\"{o}bner, universal Gr\"{o}bner, Graver) basis of $L_n$ for all $n\ge p$.
	\end{definition}
	
	\begin{remark}
		\label{relationship}
		\Cref{implications} yields the following relationship among different kinds of stabilization of a $\Sym$-invariant chain of lattices:
		\begin{align*}
			\text{Graver-stabilization } &\Rightarrow   \text{ universal Gr\"{o}bner-stabilization } 
			\Rightarrow   \text{ Gr\"{o}bner-stabilization} \\
			&\Rightarrow   \text{ Markov-stabilization } 
			\Rightarrow   \text{ stabilization.} 
		\end{align*}
	\end{remark}
	
	For equivariant Markov and Graver bases, \Cref{stabilization} can be extended as follows. 
	
	\begin{theorem}
		\label{stabilization-others}
		Let $\L=(L_n)_{n\ge1}$ be a $\Sym$-invariant chain of lattices with limit $L$. Then the following are equivalent:
		\begin{enumerate}
			\item
			$\L$ Markov- (respectively, Graver-)stabilizes;
			\item
			There exist $q, q'\in\NN$ such that for all $n\ge q$ the following hold:
			\begin{enumerate}
				\item
				$L_{n} = L \cap \ZZ^{(I_n)}$,
				\item
				$L_{n}$ has a Markov (respectively, Graver) basis consisting of elements of support size at most $q'$;
			\end{enumerate}
			\item
			$L$ has a finite equivariant Markov (respectively, Graver) basis.
		\end{enumerate}
	\end{theorem}
	
	\begin{proof}
		(i) $\Rightarrow$ (iii): 
		This follows analogously to the implication (ii) $\Rightarrow$ (iv) in the proof of \Cref{stabilization}. Note that if $\L$ Markov-stabilizes and $\M_p\subseteq L_p$ is an equivariant Markov basis of $L_n$ for $n\ge p$, then we may assume $\M_p$ is finite by using \Cref{bases_algebraic}(ii) and Hilbert's basis theorem. Similarly, if $\L$ Graver-stabilizes and $\G_p\subseteq L_p$ is an equivariant Graver basis of $L_n$ for $n\ge p$, then $\G_p$ must be finite by e.g. Gordan--Dickson lemma (see \cite[Lemma 2.5.6]{DHK}).
		
		(iii) $\Rightarrow$ (ii): 
		We obtain (ii)(a) by using \Cref{stabilization,implications}. So it remains to prove (ii)(b). 
		
		Assume first that $L$ has a finite equivariant Markov basis $\M$. Choose $m\in\NN$ with $\M\subseteq L_m$. It suffices to show that $\Sym(n)(\M)$ is a Markov basis of $L_n$ for all $n\ge p=dm+1.$ To this end, take any $\ub\in\ZZ_{\ge0}^{(I_n)}$ and let $\vb,\wb\in F_{L_n}(\ub).$ Writing $\vb-\wb=(\vb-\wb)^+-(\vb-\wb)^-$, we may assume that $\vb$ and $\wb$ have disjoint supports. In particular, $\supp(\wb)\subseteq\supp(\vb-\wb)$. Since $\vb,\wb\in F_{L}(\ub)$ and $\Sym(\M)$ is a Markov basis of $L$, there exist $\mb_1,\dots,\mb_s\in\pm\M$ and $\sigma_1,\dots,\sigma_s\in\Sym$ such that
		\[
		\vb-\wb=\sum_{i=1}^s\sigma_i(\mb_i)
		\quad\text{and}\quad
		\wb+\sum_{i=1}^t\sigma_i(\mb_i)\in F_L(\ub) \ \text{ for all }\ 
		t\in [s].
		\]
		Applying \Cref{pull_permutation} to the first equation (as in the proof of \Cref{stabilization}), we find $\sigma\in\Sym$ and $\tau_1,\dots,\tau_s\in \Sym(n)$ such that 
		\[
		\sigma(\vb-\wb) = \vb-\wb
		\quad\text{and}\quad
		\sigma\circ\sigma_i(\mb_i)=\tau_i(\mb_i)
		\ \text{ for all }\ i\in [s].
		\]
		It holds furthermore that $\sigma(\wb) = \wb$ since $\supp(\wb)\subseteq\supp(\vb-\wb)$. Therefore,
		\[
		\vb-\wb=\sigma(\vb-\wb)
		=\sum_{i=1}^s\sigma\circ\sigma_i(\mb_i)
		=\sum_{i=1}^s\tau_i(\mb_i)
		\]
		and
		\[
		\wb+\sum_{i=1}^t\tau_i(\mb_i)
		=\sigma\Big(\wb+\sum_{i=1}^t\sigma_i(\mb_i)\Big)
		\in \sigma(F_L(\ub))\cap\ZZ^{(I_n)}
		\subseteq\ZZ_{\ge0}^{(I_n)}
		\ \text{ for all }\ t\in [s].
		\]
		This confirms that $\Sym(n)(\M)$ is indeed a Markov basis of $L_n$ for all $n\ge p.$
		
		Next, suppose that $L$ has a finite equivariant Graver basis $\G$. Then we can find a $q\in\NN$ such that $\G\subseteq L_q$ and $L_{n} = L \cap \ZZ^{(I_n)}$ for all $n\ge q$. By \Cref{prop_local_global}(ii), $\Sym(\G)\cap \ZZ^{(I_n)}$ is the Graver basis of $L_n$ for $n\ge q$. Since $\Sym(\G)\cap \ZZ^{(I_n)}=\Sym(n)(\G)$ for $n\ge q$, we are done.
		
		(ii) $\Rightarrow$ (i): 
		Let $p=\max\{q,dq'+1\}$. First suppose that $L_{n}$ has a Markov basis $\M_n$ that consists of elements of support size at most $q'$ for $n\ge p$. Then arguing as in the proof of \Cref{stabilization} we find a subset $\M_n'\subseteq L_p$ such that $\M_n\subseteq \Sym(n)(\M_n')$ for $n\ge p$. In particular, $\M_n'$ is an equivariant Markov basis of $L_n$ for $n\ge p$. It follows that
		\[
		\M=\bigcup_{n\ge p}\M_n'
		\subseteq L_p
		\]
		is an equivariant Markov basis of $L_n$ for all $n\ge p$.
		
		The argument for Graver bases is similar. Note that if $\ub\in\ZZ^{(I)}$ is $\sqsubseteq$-minimal, then $\sigma(\ub)$ is also $\sqsubseteq$-minimal for any $\sigma\in \Sym.$ For $n\ge p$, let $\G_n$ be the Graver basis of $L_n$ and let $\G_n'\subseteq L_p$ be the set constructed as in the proof of \Cref{stabilization} such that $\G_n\subseteq \Sym(n)(\G_n')$. By construction, all the elements of $\G_n'$ are $\sqsubseteq$-minimal because they are of the form $\sigma(\ub)$ with $\sigma\in\Sym(n)$ and $\ub\in\G_n$. So if we set
		\[
		\G=\bigcup_{n\ge p}\G_n'
		\subseteq L_p,
		\]
		then for $n\ge p$, $\Sym(n)(\G)$ is a subset of $L_n$ that consists of $\sqsubseteq$-minimal elements. This means that $\Sym(n)(\G)\subseteq\G_n$ for $n\ge p$. On the other hand, $\G_n\subseteq \Sym(n)(\G)$ since $\G_n'\subseteq \G$. Hence, $\Sym(n)(\G)=\G_n$ and $\G$ is an equivariant Graver basis of $L_n$ for all $n\ge p$.
	\end{proof}
	
	We do not know whether \Cref{stabilization} can also be fully extended to equivariant (universal) Gr\"{o}bner bases. Nevertheless, at least a partial extension is possible.
	
	\begin{proposition}
		\label{stabilization-Groebner}
		Consider the following conditions for a $\Sym$-invariant chain of lattices $\L=(L_n)_{n\ge1}$ with limit $L$: 
		\begin{enumerate}
			\item
			$\L$ Gr\"{o}bner- (respectively, universal Gr\"{o}bner-)stabilizes.
			\item
			There exist $q, q'\in\NN$ such that for all $n\ge q$ the following hold:
			\begin{enumerate}
				\item
				$ L_{n} = L \cap \ZZ^{(I_n)}$,
				\item
				$L_{n}$ has a Gr\"{o}bner (respectively, universal Gr\"{o}bner) basis consisting of elements of support size at most $q'$.
			\end{enumerate}
			\item
			$L$ has a finite equivariant Gr\"{o}bner (respectively, universal Gr\"{o}bner) basis.
		\end{enumerate}
		Then it holds that
		\[
		\text{(i) }
		\Leftrightarrow 
		\text{(ii) }
		\Rightarrow
		\text{(iii)}.
		\]
		
		Moreover, for (equivariant) Gr\"{o}bner bases with respect to the lexicographic order $\prec_{\lex}$, one has 
		$\text{(i) }\Leftrightarrow \text{(ii) }\Leftrightarrow
		\text{(iii)}.$
	\end{proposition}
	
	\begin{proof}
		The implications (i) $\Rightarrow$ (iii) and (ii) $\Rightarrow$ (i) follow analogously to the corresponding implications in the proof of \Cref{stabilization-others}. For the implication (i) $\Rightarrow$ (ii), it suffices to prove (i) $\Rightarrow$ (ii)(b) by \Cref{stabilization,relationship}. But this is clear because if $\B_p\subseteq L_p$ is an equivariant (universal) Gr\"{o}bner basis of $L_n$ for $n\ge p$, then for such $n$, $\Sym(n)(\B_p)$ is a (universal) Gr\"{o}bner basis of $L_n$.
		
		Finally, the implication (iii) $\Rightarrow$ (ii) for equivariant Gr\"{o}bner bases with respect to $\prec_{\lex}$ follows from \Cref{local_global_Groebner}, using an argument similar to that for Graver bases in the proof of the implication (iii) $\Rightarrow$ (ii) of \Cref{stabilization-others}.
	\end{proof}
	
	\begin{remark}
	\label{rm-ideal-stabilization}
	    The action of $\Sym$ on $I$ induces actions of this group on the polynomial ring $R=K[x_{\ib}\mid \ib\in I]$ and Laurent polynomial ring $R^\pm=R[x_{\ib}^{-1}\mid \ib\in I]$ that are determined by 
	    \[
	    \sigma(x_{\ib})=x_{\sigma(\ib)}
	    \ \text{ and }\
	    \sigma(x_{\ib}^{-1})=x_{\sigma(\ib)}^{-1}
	    \ \text{ for any }  \sigma\in\Sym
	    \text{ and } \ib\in I.
	    \]
	    Analogously to lattices, one defines $\Sym$-invariant chains of (Laurent) ideals and their stabilization. Then for any $\Sym$-invariant chain of lattices $\L=(L_n)_{n\ge1}$, its stabilization and Markov-stabilization can be characterized in algebraic terms as follows:
		\begin{enumerate}
			\item 
			$\L$ stabilizes if and only if the chain of Laurent ideals $(\Ik^{\pm}_{L_n})_{n\ge1}$ stabilizes;
			\item
			$\L$ Markov-stabilizes if and only if the chain of lattice ideals $(\Ik_{L_n})_{n\ge1}$ stabilizes.
		\end{enumerate}
		See \cite{HM} and \cite{HS12} for details in the cases $c=1$ (i.e. $I=\NN^d$) or $d=1$ (i.e. $I=\NN\times[c]$), which can be easily extended to our setting. One might ask whether there is a similar interpretation for the Gr\"{o}bner-stabilization. Here, one obstruction is that for a given term order $\prec$, the chain of initial ideals $(\ini_\prec(\Ik_{L_n}))_{n\ge1}$ might be not $\Sym$-invariant. As a remedy, one replaces the action of the symmetric group $\Sym$ by that of the monoid $\Inc$ of increasing functions. Then for a term order $\prec$ such that the chain $(\ini_\prec(\Ik_{L_n}))_{n\ge1}$ is $\Inc$-invariant, the Gr\"{o}bner-stabilization with respect to $\prec$ of $\L$ can be characterized in terms of the $\Inc$-stabilization of the chain $(\ini_\prec(\Ik_{L_n}))_{n\ge1}$; see \cite{HS12,NR17}.
	\end{remark}
	
	To conclude this section, let us derive a local-global result for equivariant Hilbert bases of monoids in the flavor of \Cref{stabilization}. As for lattices, a chain of monoids $\Mk=(M_n)_{n\ge1}$ with $M_n\subseteq \ZZ^{(I_n)}$ for $n\ge 1$ is called \emph{$\Sym$-invariant} if it is increasing and each monoid $M_n$ is $\Sym(n)$-invariant. In this case, we say that $\Mk$ \emph{stabilizes} if there exists $p\in\NN$ such that
	\[
	M_n=\ZZ_{\ge0}\Sym(n)(M_m)
	\ \text{ for all } n\ge m\ge p.
	\]
	The following result generalizes \cite[Corollary 5.13]{KLR} and \cite[Lemma 5.1]{LR21}.
	\begin{theorem}
		\label{stabilization-monoid}
		Let $\Mk=(M_n)_{n\ge1}$ be a $\Sym$-invariant chain of nonnegative monoids, i.e. $M_n\subseteq \ZZ_{\ge0}^{(I_n)}$ for all $n\ge 1$. Denote $M=\bigcup_{n\ge1}M_n$. Then the following are equivalent:
		\begin{enumerate}
			\item
			$\Mk$ stabilizes and $M_n$ is finitely generated for $n\gg0$;
			\item
			There exist $p\in\NN$ and a finite equivariant Hilbert basis $\Hc_p$ of $M_p$ such that $\Hc_p$ is also an equivariant Hilbert basis of $M_n$ for all $n\ge p$;
			\item
			There exist $q, q'\in\NN$ such that for all $n\ge q$ the following hold:
			\begin{enumerate}
				\item
				$M_{n} = M \cap \ZZ^{(I_n)}$,
				\item
				$M_{n}$ has a finite Hilbert basis consisting of elements of support size at most $q'$;
			\end{enumerate}
			\item
			$M$ has a finite equivariant Hilbert basis.
		\end{enumerate}
	\end{theorem}
	
	\begin{proof}
		Observe that if $\ub\in M$ is an irreducible element, then $\sigma(\ub)$ is also an irreducible element of $M$ for any $\sigma\in\Sym.$
		
		(i) $\Rightarrow$ (ii): If $\Mk$ stabilizes, there exists $p\in\NN$ such that $M_n$ is generated by $\Sym(n)(M_p)$ for all $n\ge p$. Choosing $p$ large enough, we may assume that $M_p$ is finitely generated. So if $\Hc_p$ is any equivariant Hilbert basis of $M_p$, then $\Hc_p$ is finite and the set
		$$\Sym(n)(\Sym(p)(\Hc_p))=\Sym(n)(\Hc_p)$$
		generates $M_n$ for all $n\ge p$. Since $\Sym(n)(\Hc_p)$ consists of irreducible elements, it must be the Hilbert basis of $M_n$ by \Cref{lem_Hilbert_monoid}.
		
		(iv) $\Rightarrow$ (iii): To show that $M_{n} = M \cap \ZZ^{(I_n)}$ for $n\gg0$ one can argue similarly to the proof of \Cref{stabilization}. However, employing the nonnegativity of $\Mk$, this can be proven without using \Cref{pull_permutation}. Let $\Hc$ be a finite equivariant Hilbert basis of $M$ and choose $q\in\NN$ such that $\Hc\subseteq M_q$. We show that $M_{n} = M \cap \ZZ^{(I_n)}$ for all $n\ge q$. Indeed, take any $\ub\in M \cap \ZZ^{(I_n)}$ with $n\ge q$. Since $M=\ZZ_{\ge0}\Sym(\Hc)$, there exist $\hb_1,\dots,\hb_s\in\Hc$, $m_1,\dots,m_s\in\ZZ_{\ge0}$ and $\sigma_1,\dots,\sigma_s\in\Sym$ such that $\ub=\sum_{j=1}^sm_j\sigma_j(\hb_j).$ We may assume that all the coefficients $m_j$ are positive. Then the nonnegativity of the elements $\hb_j$ yields
		\[
		\supp(\sigma_j(\hb_j))\subseteq\supp(\ub)\subseteq I_n
		\ \text{ for all }\ j\in [s].
		\]
		It follows that $\sigma_j(\hb_j)=\tau_j(\hb_j)$ with $\tau_j\in\Sym(n)$ for all $j\in [s]$. Hence, 
		$$\ub=\sum_{j=1}^sm_j\tau_j(\hb_j)\in \ZZ_{\ge0}\Sym(n)(M_q)\subseteq M_n,$$
		and therefore $M_{n} = M \cap \ZZ^{(I_n)}$, as desired.
		
		Now using \Cref{local_global_Hilbert} we see that $\Sym(\Hc)\cap M_n=\Sym(n)(\Hc)$ is the Hilbert basis of $M_n$ for $n\ge q.$
		
		Finally, the implications (ii) $\Rightarrow$ (iv) and (iii) $\Rightarrow$ (i) follow similarly to the corresponding implications in the proof of \Cref{stabilization}.
	\end{proof}
	
	\section{Finiteness of equivariant Graver bases}
	\label{sec-equi-Graver}
	This section deals with the case $d=1$. Our goal is to prove the following finiteness result for equivariant Graver bases and discuss its consequences. Throughout the section, let $c$ be a positive integer.
	
	\begin{theorem}
		\label{Finiteness-Graver}
		Suppose that $I=\NN\times[c]$. Then every $\Sym$-invariant lattice in $\ZZ^{(I)}$ has a finite equivariant Graver basis.
	\end{theorem}
	
	Before proving this result, let us derive some corollaries. First, using \Cref{bases_algebraic} we get the following algebraic version of \Cref{Finiteness-Graver}.
	
	\begin{corollary}
	    Let $I=\NN\times[c]$. Then for any $\Sym$-invariant lattice $L\subseteq\ZZ^{(I)}$, the lattice ideal $\Ik_L$ has a finite equivariant Graver basis.
	\end{corollary}
	
	 The next consequence is immediate from \Cref{Finiteness-Graver,implications}.
	
	\begin{corollary}
		\label{cor-finiteness-others}
		Let $I=\NN\times[c]$. Then every $\Sym$-invariant lattice in $\ZZ^{(I)}$ has
		\begin{enumerate}
			\item 
			a finite equivariant universal Gr\"{o}bner basis,
			\item 
			a finite equivariant Gr\"{o}bner basis with respect to an arbitrary term order on $\ZZ_{\ge0}^{(I)}$,
			\item 
			a finite equivariant Markov basis,
			\item 
			a finite equivariant generating set.
		\end{enumerate}
	\end{corollary}
	
	As a local version of \Cref{Finiteness-Graver,cor-finiteness-others}, the result below is obtained by combining Theorems \ref{Finiteness-Graver}, \ref{stabilization-others} and \Cref{relationship}.
	
	\begin{corollary}
		\label{cor-stabilizations}
		Let $I=\NN\times[c]$. Then every $\Sym$-invariant chain of lattices in $\ZZ^{(I)}$ Graver-stabilizes. In particular, such a chain
		\begin{enumerate}
			\item 
			universal Gr\"{o}bner-stabilizes,
			\item 
			Gr\"{o}bner-stabilizes with respect to an arbitrary term order on $\ZZ_{\ge0}^{(I)}$,
			\item 
			Markov-stabilizes,
			\item 
			stabilizes.
		\end{enumerate}
	\end{corollary}

	This corollary generalizes and strengthens a finiteness result for Markov bases of hierarchical models due to Ho\c{s}ten and Sullivant \cite[Theorem 1.1]{HoS}. To present this result, let us recall some necessary notions; see, e.g. \cite[Chapter 9]{Su} for further details. Hierarchical models are used in statistics to represent interactions among random variables. Such a model is determined by a simplicial complex $\Delta\subseteq 2^{[m]}$ and a vector $\rb=(r_1,\dots,r_m)\in\NN^m$, where $m\in\NN$ is a given number. While the simplicial complex $\Delta$ describes the interactions among the random variables, the vector $\rb$ encodes their numbers of states. Let $F_1,\dots,F_s$ be the facets of $\Delta$. Denote $\R=\prod_{i=1}^m[r_i]$ and $\R_{F_k}=\prod_{j\in F_k}[r_j]$ for $k\in[s]$. For $\ib=(i_1,\dots,i_m)\in \R$ let $\ib_{F_k}=(i_j)_{j\in F_k}\in \R_{F_k}$. The \emph{$F_k$-marginal} of each element of $\ZZ^{(\R)}$ can be computed using the linear map 
	$\mu_{F_k}:\ZZ^{(\R)}\to \ZZ^{(\R_{F_k})}$, which 
	is defined by $\mu_{F_k}(\eb_\ib)=\eb_{\ib_{F_k}}$, where $\eb_\ib$ and $\eb_{\ib_{F_k}}$ are standard basis elements of $\ZZ^{(\R)}$ and $ \ZZ^{(\R_{F_k})}$, respectively. The minimal sufficient statistics of the model are then given by the \emph{$\Delta$-marginal map} that puts all facet marginals together:
	\[
	\mu_{\Delta,\rb}:\ZZ^{(\R)}\to \bigoplus_{k=1}^s\ZZ^{(\R_{F_k})},
	\ \ \ub\mapsto (\mu_{F_1}(\ub),\dots,\mu_{F_s}(\ub)).
	\]
	One is interested in Markov bases of the lattice $\ker \mu_{\Delta,\rb}$ as they are useful for performing Fisher's exact test. In particular, it is desirable to know how the structure of these Markov bases depends on $\Delta$ and $\rb$. Ho\c{s}ten and Sullivant \cite{HoS} studied this problem when $\Delta$ and the numbers $r_2,\dots,r_m$ are fixed, while $r_1$ is allowed to vary. Their remarkable finiteness result \cite[Theorem 1.1]{HoS} says that, in such case, there is always a finite Markov basis up to symmetry. To interpret this result in our language, we identify $\R$ with $[r_1]\times [c]$, where $c=\prod_{i=2}^mr_i$. In this way, $\ker \mu_{\Delta,\rb}$ is identified with a lattice $L_{\Delta,\rb}$ in $\ZZ^{([r_1]\times[c])}$. Consider the action of $\Sym(r_1)$ on $\ZZ^{([r_1]\times[c])}$ by permuting the first index as before. Then it is easy to see that the chain of lattices $\L_{\Delta,\rb}=(L_{\Delta,\rb})_{r_1\ge1}$ is $\Sym$-invariant. The result of Ho\c{s}ten and Sullivant can now be stated as a special case of \Cref{cor-stabilizations}(iii) as follows.
	
	\begin{corollary}
		Keep the notation as above. If $\Delta$ and $r_2,\dots,r_m$ are fixed, then the chain of lattices $\L_{\Delta,\rb}=(L_{\Delta,\rb})_{r_1\ge1}$ Markov-stabilizes. 
	\end{corollary}
	
	It should be mentioned that this result generalizes finiteness results for the no three-way interaction model due to Aoki and Takemura \cite{AT} and for logit models due to Santos and Sturmfels \cite{SSt}. A far-reaching generalization of this corollary, the so-called independent set theorem, will be discussed in the next section (see \Cref{independent-set}). 
	
	Let us now turn to the proof of \Cref{Finiteness-Graver}. We first discuss the case $c=1$ before giving a proof for the general case.
	
	\subsection{The case \texorpdfstring{$c=1$}{c=1}}
	We consider the case $c=1$ separately for two reasons. First, in this case, we obtain more specific information about the Graver basis. Second, the proof of this case is more elementary, in the sense that it does not make use of Higman's lemma. It should be mentioned that Higman's lemma is essential to all known proofs of the equivariant Noetherianity of the polynomial ring $K[x_i\mid i\in\NN]$; see \cite[Remark 15]{HM}. 
	
	As before, we denote the standard basis of $\ZZ^{(\NN)}$ by $\{\eb_i\}_{i\in\NN}$. For $\ub=(u_i)_{i\in\NN}\in\ZZ^{(\NN)}$ let 
	$$g(\ub)=\gcd(u_i\mid i\in\NN)$$
	be the greatest common divisor of the components of $\ub$. Moreover, for a lattice $L\subseteq\ZZ^{(\NN)}$ consider the element
	\begin{equation}
		\label{eq-g}
		\gb_L=(g_L,-g_L)=g_L(\eb_1-\eb_2)\in \ZZ^{(\NN)},
		\quad\text{where }\
		g_L=\gcd(g(\ub)\mid \ub\in L).
	\end{equation}
	The role of this element will be made clear in \Cref{thm_Graver}. At first, we observe:
	
	\begin{lemma}
		\label{lem_element}
		Let $L\subseteq\ZZ^{(\NN)}$ be a $\Sym$-invariant lattice and let $\ub=(u_i)_{i\in\NN}\in L.$
		Then the following statements hold:
		\begin{enumerate}
			\item 
			$(u_i,-u_i)\in L$ for all $i\in\NN$.
			\item
			$\gb_L\in L$.
		\end{enumerate}
	\end{lemma}
	
	\begin{proof}
		(i) Note that $u_n=0$ for $n\gg0$. So for any $i\in\NN$, there exist $\sigma_1,\sigma_2\in\Sym$ such that
		\[
		\sigma_1(\ub)=(u_i,0,u_1,u_2,\dots,u_{i-1},u_{i+1},\dots),\
		\sigma_2(\ub)=(0,u_i,u_1,u_2,\dots,u_{i-1},u_{i+1},\dots).
		\]
		Since $L$ is $\Sym$-invariant, we get
		\[
		(u_i,-u_i)=\sigma_1(\ub)-\sigma_2(\ub)\in L.
		\]
		
		(ii) By definition of $g(\ub)$, there exist $i_1,\dots,i_k\in\NN$ and $z_{i_1},\dots,z_{i_k}\in\ZZ$ such that
		$$
		g(\ub)=\sum_{j=1}^kz_{i_j}u_{i_j}.
		$$
		Thus from (i) it follows that 
		$$(g(\ub),-g(\ub))=\sum_{j=1}^kz_{i_j}(u_{i_j},-u_{i_j})\in L.$$ Similarly, by definition of $g_L$, there exist $\vb_1,\dots,\vb_l\in L$ and $y_{1},\dots,y_{l}\in\ZZ$ such that
		\[
		g_L=\sum_{j=1}^ly_{j}g(\vb_{j}).
		\]
		Hence, 
		$$
		\gb_L=(g_L,-g_L)=\sum_{j=1}^ly_{j}(g(\vb_{j}),-g(\vb_{j}))\in L,
		$$
		as desired.
	\end{proof}
	
	As a consequence of the preceding lemma, the next result shows an interesting property of \emph{generic} $\Sym$-invariant lattices in $\ZZ^{(\NN)}$: they always contain certain multiples of standard basis elements.
	
	\begin{corollary}
	\label{Graver-element}
		Let $L\subseteq\ZZ^{(\NN)}$ be a $\Sym$-invariant lattice. For any $\ub=(u_i)_{i\in\NN}\in \ZZ^{(\NN)}$ set $s(\ub)=\sum_{i\in\NN}u_i$. Denote $s_L=\gcd(s(\ub)\mid \ub\in L)$. Then the following hold:
		\begin{enumerate}
			\item 
			$s_L\eb_1\in L.$
			\item
			If there exists $\ub\in L$ such that $s(\ub)\ne0$, then $s_L\eb_1$ belongs to the Graver basis of $L$.
		\end{enumerate}  
	\end{corollary}

	\begin{proof}
		(i) Take any $\ub\in L$ and choose $n$ such that $u_i=0$ for all $i\ge n$, i.e. $\ub=(u_1,\dots,u_n)$.
		\Cref{lem_element} implies $\fb=u_n(\eb_{n-1}-\eb_n)=(0,\dots,0,u_{n},-u_{n})\in L.$
		Hence, 
		\[
		\ub'=\ub+\fb
		=(u_1,\dots,u_{n-2},u_{n-1}+u_{n})
		\in L.
		\]
		Repeating the above argument, we see that $s(\ub)\eb_1\in L.$ Since $s_L$ is a $\ZZ$-linear combination of finitely many $s(\ub)$ with $\ub\in L$, the desired conclusion follows.
		
		(ii) If $s(\ub)\ne0$ for some $\ub\in L$, then $s_L\ne 0$. In this case, $s_L\eb_1$ is a $\sqsubseteq$-minimal element of $L\setminus\{\nub\}$. Indeed, let $\vb\in L\setminus\{\nub\}$ with $\vb\sqsubseteq s_L\eb_1$. Then $\vb$ must be of the form $z\eb_1$ for some $z\in\ZZ$ with $0<z\le s_L$. Since $s(\vb)=z$ is divisible by $s_L$, we deduce that $z=s_L$, and thus $\vb=s_L\eb_1$, as desired.
	\end{proof}
	
	In order to determine the Graver basis of a $\Sym$-invariant lattice $L\subseteq\ZZ^{(\NN)}$, we relate it to the Hilbert basis of the monoid $M=L\cap\ZZ^{(\NN)}_{\ge0}$; see \Cref{lem_Hilbert_monoid,prop_Hilbert_Graver} for analogous ideas. The following key lemma is reminiscent of the equivariant Gordan's lemma established in \cite[Theorem 6.1]{LR21}.
	
	\begin{lemma}
		\label{lem_lattice_monoid}
		For every $\Sym$-invariant lattice $L\subseteq\ZZ^{(\NN)}$, the monoid $M=L\cap\ZZ^{(\NN)}_{\ge0}$ has a finite equivariant Hilbert basis.
	\end{lemma}
	
	\begin{proof}
		Set $L_n=L\cap \ZZ^n$ and $M_n=L_n\cap\ZZ^{n}_{\ge0}$ for $n\ge1$. Moreover, let $\Hc_n$ denote the Hilbert basis of $M_n$. Obviously, $\Mk=(M_n)_{n\ge1}$ is a $\Sym$-invariant chain of nonnegative monoids with limit $\bigcup_{n\ge1}M_n=M$. So by \Cref{stabilization-monoid}, it suffices to show that each $\Hc_n$ is finite and the support sizes of its elements do not exceed a fixed number $q$. First of all, $L_n$ is a finitely generated monoid as it is a finitely generated abelian group. Hence, $M_n=L_n\cap\ZZ^{n}_{\ge0}$ is also a finitely generated monoid; see, e.g. \cite[Corollary 2.11(a)]{BG}. This implies that
		$\Hc_n$ is finite for all $n\ge 1$. Set
		\[
		\|\Hc_{n}\|=\max\{\|\ub\|\mid \ub\in \Hc_{n}\}.
		\]
		Evidently, $|\supp(\ub)|\le \|\ub\|$ for all $\ub\in\ZZ^{(\NN)}$. Thus, in order to verify that the support sizes of elements of $\Hc_n$ are uniformly bounded, we only need to show that
		$\|\Hc_{n+1}\|\le \|\Hc_{n}\|$ 
		for all $n\ge 1.$ To this end, take any 
		$\ub=(u_1,\dots,u_{n+1})\in \Hc_{n+1}$. From \Cref{lem_element} it follows that
		\[
		\fb=u_{n+1}(\eb_n-\eb_{n+1})=(0,\dots,0,u_{n+1},-u_{n+1})\in L.
		\]
		Hence,
		\[
		\ub'=\ub+\fb
		=(u_1,\dots,u_{n-1},u_{n}+u_{n+1})
		\in L\cap\ZZ^{n}_{\ge0} = M_{n}.
		\]
		Since $\|\ub\|=\|\ub'\|$, the desired inequality will follow if we show that $\ub'\in\Hc_{n}.$ Assume the contrary that $\ub'\not\in\Hc_{n}$. Then there exist 
		$\vb'=(v_1,\dots,v_{n}),\wb'=(w_1,\dots,w_{n})\in M_{n}\setminus\{\nub\}$ such that $\ub'=\vb'+\wb'.$ This gives $u_{n}+u_{n+1}=v_{n}+w_{n}.$ We may suppose $u_{n}\le v_{n}.$ Since $\ub-\vb'\in L$, \Cref{lem_element} yields
		\[
		\fb'=(u_n-v_{n})(\eb_n-\eb_{n+1})
		=(0,\dots,0,u_n-v_{n},v_{n}-u_{n})
		\in L.
		\]
		It follows that
		\[
		\vb=\vb'+\fb'
		=(v_1,\dots,v_{n-1},u_{n},v_{n}-u_{n})
		\in L\cap\ZZ^{n+1}_{\ge0} = M_{n+1}.
		\]
		On the other hand, $\wb=(w_1,\dots,w_{n-1},0,w_{n})\in M_{n+1}$ since $\wb'=(w_1,\dots,w_{n},0)\in M_{n+1}.$ Now from $\ub'=\vb'+\wb'$ we get $\ub=\vb+\wb$, contradicting the assumption that $\ub\in \Hc_{n+1}$.
	\end{proof}
	
	We are now ready to prove the following more specific version of \Cref{Finiteness-Graver} in the case $c=1$.
	
	\begin{theorem}
		\label{thm_Graver}
		Let $L\subseteq\ZZ^{(\NN)}$ be a $\Sym$-invariant lattice. Suppose that $\Hc$ is a finite equivariant Hilbert basis of the monoid $M=L\cap\ZZ^{(\NN)}_{\ge0}$. Then $L$ has a finite equivariant Graver basis $\G$ satisfying
		\[
		\pm\Hc\subseteq\G\subseteq\pm\Hc\cup\{\pm\gb_L\},
		\]
		where $\gb_L$ is defined as in \eqref{eq-g}.
	\end{theorem}

	\begin{proof}
		First, note that $\Hc$ exists by \Cref{lem_lattice_monoid}. Moreover, it follows from \Cref{lem_Hilbert_monoid} that every element of $\pm\Hc$ is $\sqsubseteq$-minimal in $L\setminus\{\nub\}$. So to prove the theorem, it suffices to show that $\Sym(\Hc')$ contains all $\sqsubseteq$-minimal elements of $L\setminus\{\nub\}$, where $\Hc'=\pm\Hc\cup\{\pm\gb_L\}$. Take
		any $\ub\in L\setminus\{\nub\}$. We need to prove  that $\sigma(\hb)\sqsubseteq \ub$ for some $\sigma\in\Sym$ and $\hb\in\Hc'.$ 
		
		If $\ub\in M$ or $-\ub\in M$, then we are done since $\Hc$ is an equivariant Hilbert basis of $M$. 
		Now suppose $\ub=(u_i)_{i\in\NN}\not\in\pm M$. Then $\supp(\ub^+),\supp(\ub^-)\ne\emptyset$. Replacing $\ub$ with $\pm\tau(\ub)$ for a suitable $\tau\in\Sym$ (if necessary), we may assume that $1\in\supp(\ub^+)$, $2\in\supp(\ub^-)$ and $|u_1|\le |u_2|$. Then $\ub=\vb+\wb$  is a conformal decomposition of $\ub$, where
		\[
		\vb=(u_1,-u_1)
		\quad\text{and}\quad
		\wb=(0,u_1+u_2,u_3,\dots).
		\]
		From the definition of $\gb_L$ it is clear that $\vb=z\gb_L$ for some $z\in\ZZ\setminus\{0\}$. Hence, we have either $\gb_L\sqsubseteq\vb\sqsubseteq\ub$ or $-\gb_L\sqsubseteq\vb\sqsubseteq \ub.$ This completes the proof.
	\end{proof}
	
	\begin{example}
		Consider the lattice $L=\ZZ\Sym(\{\ub,\vb\})\subseteq\ZZ^{(\NN)}$ generated by the $\Sym$-orbits of the elements
		\[
		\ub=(1,3,5)\ \text{ and }\ \vb= (2,4,6).
		\]
		Let $M=L\cap\ZZ^{(\NN)}_{\ge0}$. Denote by $\G$ an equivariant Graver basis of $L$. Then it follows from \Cref{lem_Hilbert_monoid} that $\Hc=\G\cap\ZZ^{(\NN)}_{\ge0}$ is an equivariant Hilbert basis of $M$. By \Cref{Graver-element}, we may assume that $3\eb_1\in \G$ since $s_L=\gcd(s(\ub),s(\vb))=3.$ Hence, $3\eb_1\in \Hc$. So the proof of \Cref{lem_lattice_monoid} implies that we may choose $\Hc\subseteq\{3\eb_1,\eb_1+2\eb_2\}$. On the other hand, by \Cref{thm_Graver}, $\gb_L=\eb_1-\eb_2$ could be an element of $\G$. Computations with 4ti2 \cite{4ti2} and Macaulay2 \cite{GS} confirm that we can take $\Hc=\{3\eb_1,\eb_1+2\eb_2\}$ and $\G=\pm\Hc\cup\{\pm\gb_L\}$.
		
		Note that, in general, $\gb_L$ is not necessarily an element of the Graver basis of $L$. An obvious example is $L=\ZZ^{(\NN)}$, of which $\G=\{\pm\eb_1\}$ is an equivariant Graver basis.
	\end{example}
	
	\subsection{The general case}
	\label{sec-Graver}
	A key step in the proof of \Cref{Finiteness-Graver}, as in the case $c=1$, is to relate the Graver basis of a lattice to the Hilbert basis of a certain monoid. We present this relationship in the general setting as it might be of independent interest.
	
	Let $I$ be an arbitrary set. Denote by $2I$ the disjoint union $I\sqcup I$. Then there is a natural isomorphism of abelian groups
	$$\psi:\ZZ^{(2I)}\cong \ZZ^{(I)}\oplus\ZZ^{(I)}.$$
	Identifying $\ZZ^{(2I)}$ with $\ZZ^{(I)}\oplus\ZZ^{(I)}$ using this isomorphism, we write each element of $\ZZ^{(2I)}$ in the form $(\ub,\vb)$ with $\ub,\vb\in\ZZ^{(I)}$. Consider the map
	\[
	\varphi: \ZZ^{(2I)}\to \ZZ^{(I)},\ 
	(\ub,\vb)\mapsto \ub-\vb.
	\]
	Via this map, the Graver basis of any lattice in $\ZZ^{(I)}$ is intimately related to the Hilbert basis of a nonnegative monoid in $\ZZ^{(2I)}$. As usual, we denote standard basis of $\ZZ^{(I)}$ by $\{\eb_\ib\}_{\ib\in I}$.

	\begin{proposition}
		\label{prop_Hilbert_Graver}
		Let $L\subseteq\ZZ^{(I)}$ be a lattice and set $M=\varphi^{-1}(L)\cap \ZZ^{(2I)}_{\ge0}$. Denote by $\G$ the Graver basis of $L$ and $\Hc$ the Hilbert basis of $M$. Then $L=\varphi(M)$. Moreover, the following hold:
		\begin{enumerate}
			\item 
			$\G=\varphi(\Hc)\setminus\{\nub\}$.
			\item
			$\Hc=\{(\ub^+,\ub^-)\mid \ub=\ub^+-\ub^-\in \G\}\cup\{(\eb_\ib,\eb_\ib)\mid \ib\in I\}$.
		\end{enumerate}
	\end{proposition}
	
	\begin{proof}
		Since every element $\ub\in L$ can be written in the form $\ub=\ub^+-\ub^-$, we see that $\ub=\varphi((\ub^+,\ub^-))\in \varphi(M)$. Thus $\varphi(M)=L.$ 
		
		(i) Let $\ub\in \G$. Then $(\ub^+,\ub^-)\in M\setminus\{\nub\}$. So by \Cref{lem_Hilbert_monoid}, there exists $(\vb,\wb)\in\Hc$ such that  $(\vb,\wb)\sqsubseteq (\ub^+,\ub^-)$. This means that $\vb\sqsubseteq \ub^+$ and $\wb\sqsubseteq \ub^-$. In particular, $\vb\ne\wb$ as $\ub^+$ and $\ub^-$ have disjoint supports. Hence,
		\[
		\nub\ne\varphi((\vb,\wb))= \vb-\wb
		\sqsubseteq \ub^+-\ub^-=\ub.
		\]
		Since $\ub\in \G$, we must have $\ub=\varphi((\vb,\wb))\in\varphi(\Hc)\setminus\{\nub\}$.
		
		Conversely, take $\vb-\wb\in \varphi(\Hc)\setminus\{\nub\}$ with $(\vb,\wb)\in\Hc$. Then there exists $\ub\in\G$ such that $\ub=\ub^+-\ub^-\sqsubseteq \vb-\wb$. This implies $(\ub^+,\ub^-)\sqsubseteq (\vb,\wb)$ since $\ub^+$ and $\ub^-$ have disjoint supports. The fact that $(\vb,\wb)\in\Hc$ forces $(\ub^+,\ub^-)= (\vb,\wb)$. Hence, $\vb-\wb=\ub\in\G.$
		
		(ii) Evidently, $(\eb_\ib,\eb_\ib)\in\Hc$ for all $\ib\in I$. Moreover, arguing as in the first part of the proof of (i), we see that $(\ub^+,\ub^-)\in\Hc$ for all $\ub\in\G$.
		
		Conversely, let $(\vb,\wb)\in\Hc$. If $\vb=\wb$, then $(\eb_\ib,\eb_\ib)\sqsubseteq (\vb,\wb)$ for some $\ib\in I$, which implies $(\vb,\wb)=(\eb_\ib,\eb_\ib)$. Otherwise, if $\vb\ne\wb$, then arguing as in the second part of the proof of (i), we get $(\vb,\wb)=(\ub^+,\ub^-)$ for some $\ub\in\G$.
	\end{proof}
	
	When $I=\NN^d\times[c]$, one can identify $2I$ with $\NN^d\times[2c]$. In this case, it is apparent that the maps $\psi$ and $\varphi$ are compatible with the action of the group $\Sym$, i.e.
	\[
	\psi(\sigma(\ub))=\sigma(\psi(\ub)),\quad
	\varphi(\sigma(\ub))=\sigma(\varphi(\ub))
	\quad\text{for all }\
	\ub\in \ZZ^{(2I)}, \sigma\in\Sym.
	\]
	Thus, in particular, $\varphi^{-1}(L)$ is a $\Sym$-invariant lattice in $\ZZ^{(2I)}$ for any $\Sym$-invariant lattice $L\subseteq\ZZ^{(I)}$. 
	
	\begin{corollary}
		\label{cor_Hilbert_Graver}
		Suppose $I=\NN^d\times[c]$, where $c,d\in\NN$. Let $L\subseteq\ZZ^{(I)}$ be a $\Sym$-invariant lattice and let $M=\varphi^{-1}(L)\cap \ZZ^{(2I)}_{\ge0}$. 
		Then the following statements are equivalent:
		\begin{enumerate}
			\item 
			$M$ has a finite equivariant Hilbert basis;
			\item
			$L$ has a finite equivariant Graver basis.
		\end{enumerate} 
	\end{corollary}

	\begin{proof}
		(i)$\Rightarrow$(ii): Suppose $\Hc$ is an equivariant Hilbert basis of $M$. Since $\varphi$ is compatible with the $\Sym$-action, we have
		\[
		\Sym(\varphi(\Hc)\setminus\{\nub\})
		=
		\varphi(\Sym(\Hc))\setminus\{\nub\}.
		\]
		By \Cref{prop_Hilbert_Graver}, $\varphi(\Sym(\Hc))\setminus\{\nub\}$ is the Graver basis of $L$. Hence, $\varphi(\Hc)\setminus\{\nub\}$ is an equivariant Graver basis of $L$. If $\Hc$ is finite, then of course $\varphi(\Hc)\setminus\{\nub\}$ is also finite.
		
		(ii)$\Rightarrow$(i): Let $\G$ be an equivariant Graver basis of $L$. Denote
		\[
		\Hc=\{(\ub^+,\ub^-)\mid \ub=\ub^+-\ub^-\in \G\}
		\cup\{(\eb_\ib,\eb_\ib)\mid \ib\in [d]^d\times[c]\}.
		\]
		Then it is easy to check that
		\[
		\Sym(\Hc)=\{(\ub^+,\ub^-)\mid \ub=\ub^+-\ub^-\in \Sym(\G)\}
		\cup\{(\eb_\ib,\eb_\ib)\mid \ib\in I\}.
		\]
		Hence, $\Hc$ is an equivariant Hilbert basis of $M$ by \Cref{prop_Hilbert_Graver}. Clearly, $\Hc$ is finite if $\G$ is.
	\end{proof}
	
	In view of the previous corollary, \Cref{Finiteness-Graver} will follow from the next lemma that is a generalization of  \Cref{lem_lattice_monoid}.
	
	\begin{lemma}
		\label{lattice-monoid-ext}
		Let $I=\NN\times[c]$, where $c\in\NN$. Then for any $\Sym$-invariant lattice $L\subseteq\ZZ^{(I)}$, the monoid $M=L\cap\ZZ^{(I)}_{\ge0}$ has a finite equivariant Hilbert basis.
	\end{lemma}
	
	We do not know whether the ``elementary" proof of \Cref{lem_lattice_monoid} can be adapted to prove this result. Our proof here uses a different argument that requires Higman's lemma. For the proof,
	let us first recall a few notions. Let $\le$ be a partial order on a set $S$. Then $\le$ is called a \emph{well-partial-order} if any infinite sequence $s_1,s_2,\dots$ in $S$ is \emph{good}, i.e. $s_i\le s_j$ for some indices $i<j$. For example, $\sqsubseteq$ is a well-partial-order on $\ZZ^n$ by Gordan--Dickson lemma (see \cite[Lemma 2.5.6]{DHK}). A \emph{final segment} is a subset $\F\subseteq S$ with the property that if $s\in \F$ and $s\le t\in S$, then $t\in \F$. The final segment \emph{generated} by a subset $T\subseteq S$ is defined as
	\[
	\F(T)\defas\{s\in S\mid s\le t \text{ for some } t\in T\}.
	\]
	A classical characterization of well-partial-orders (see, e.g. \cite[Theorem 2.1]{Hig}) is that every final segment is finitely generated.
	Let $S^*=\bigcup_{n\in\NN}S^n$ denote the set of finite sequences of elements of $S$. Define the \emph{Higman order} $\le_H$ on $S^*$ as follows: $(s_1,\dots,s_p)\le_H (s_1',\dots,s_q')$ if there exists a strictly increasing map $\pi:[p]\to [q]$ such that $s_i\le s_{\pi(i)}'$ for all $i\in[p]$. Then Higman's lemma \cite[Theorem 4.3]{Hig} states that $\le_H$ is a well-partial-order on $S^*$.
	
	In order to apply Higman's lemma to the situation of \Cref{lattice-monoid-ext}, it is convenient to consider the monoid of strictly increasing functions
	\[
	\Inc = \{ \pi \colon \NN \to \NN \mid  \pi(i)<\pi(i+1) 
	\text{ for all } i\in \NN\}.
	\]
	This monoid acts on $\ZZ^{(I)}$ when $I=\NN^d\times[c]$ analogously to $\Sym$. For simplicity, we only describe the action in the case $I=\NN\times[c]$, which is enough for the proof of \Cref{lattice-monoid-ext}. Let $\{\eb_{i,j}\}_{(i,j)\in\NN\times[c]}$ be the standard basis of $\ZZ^{(I)}$. Then the action of $\Inc$ on $\ZZ^{(I)}$ is determined by
	\[
	\pi(\eb_{i,j})= \eb_{\pi(i),j}
	\ \text{ for any } \pi\in\Inc
	\text{ and } (i,j)\in\NN\times[c].
	\]
	This action is closely related to that of $\Sym$; see \cite{KLR} for details in the case $c=1$. Note that for any $\pi\in\Inc$ and $\ub\in\ZZ^{(I)}$, there exists $\sigma\in\Sym$ such that $\pi(\ub)=\sigma(\ub).$ Indeed, $\supp(\ub)$ is contained in $[n]\times[c]$ for some $n\in\NN$. Since the restriction $\pi\restr{[n]}$ is injective, there exists $\sigma\in\Sym$ such that $\pi\restr{[n]}=\sigma\restr{[n]}$, which yields $\pi(\ub)=\sigma(\ub).$
	
	We are now ready to prove \Cref{lattice-monoid-ext}.

	\begin{proof}[Proof of \Cref{lattice-monoid-ext}]
		Define a partial order $\sqsubseteq_{\Inc}$ on $\ZZ_{\ge0}^{(I)}$ as follows: $\ub\sqsubseteq_{\Inc}\vb$ if there exists $\pi\in \Inc$ such that $\pi(\ub)\sqsubseteq\vb.$ We show that $\sqsubseteq_{\Inc}$ is a well-partial-order. Indeed, recall that $\sqsubseteq$ is a well-partial-order on $\ZZ_{\ge0}^{c}$ by Gordan--Dickson lemma. So by Higman's lemma, the Higman order $\sqsubseteq_H$ is a well-partial-order on $(\ZZ_{\ge0}^{c})^*$.  Via the natural bijection
		\[
		\ZZ_{\ge0}^{(I)}=\ZZ_{\ge0}^{(\NN\times [c])}
		=\bigcup_{n\in\NN}\ZZ_{\ge0}^{([n]\times [c])}
		\cong \bigcup_{n\in\NN}(\ZZ_{\ge0}^{c})^n
		=(\ZZ_{\ge0}^{c})^*,
		\]
		we may identify the partially ordered set $(\ZZ_{\ge0}^{(I)},\sqsubseteq_{\Inc})$ with $((\ZZ_{\ge0}^{c})^*,\sqsubseteq_H)$. Therefore, $\sqsubseteq_{\Inc}$ is a well-partial-order on $\ZZ_{\ge0}^{(I)}$.
		
		Now consider the final segment generated by $M\setminus\{\nub\}$ with respect to $\sqsubseteq_{\Inc}$:
		\[
		\F=\F(M\setminus\{\nub\})=\{\vb\in \ZZ_{\ge0}^{(I)}
		\mid \ub\sqsubseteq_{\Inc}\vb \text{ for some } \ub\in M\setminus\{\nub\}\}.
		\]
		Since $\sqsubseteq_{\Inc}$ is a well-partial-order, $\F$ is finitely generated. That is, $\F=\F(\Hc)$ for some finite subset $\Hc\subseteq M\setminus\{\nub\}$.
		We prove that $\Hc$ contains an equivariant Hilbert basis of $M$. By \Cref{lem_Hilbert_monoid}, it suffices to show that $\Sym(\Hc)$ contains all $\sqsubseteq$-minimal elements of $M\setminus\{\nub\}$. Indeed, for any $\vb\in M\setminus\{\nub\}$, there exist $\ub\in\Hc$ and $\pi\in\Inc$ such that $\pi(\ub)\sqsubseteq \vb$. Since $\pi(\ub)=\sigma(\ub)$ for some $\sigma\in\Sym$, the desired assertion follows.
	\end{proof}  
	
	Let us conclude this section with the proof of \Cref{Finiteness-Graver}.
	
	\begin{proof}[Proof of \Cref{Finiteness-Graver}]
		The result follows easily from \Cref{cor_Hilbert_Graver,lattice-monoid-ext}.
	\end{proof}
	
	\section{Finiteness of equivariant generating sets and Markov bases}
	\label{sec-Markov}
	Throughout this section, let $I=\NN^d\times[c]$ with $c,d\in\NN$. We study here \Cref{main-problem} for equivariant generating sets and Markov bases of $\Sym$-invariant lattices in $\ZZ^{(I)}$. Based on a result of Hillar and Mart\'{i}n del Campo \cite[Theorem 19]{HM}, we show that every $\Sym$-invariant lattice in $\ZZ^{(I)}$ has a finite equivariant generating set. Concerning finiteness of equivariant Markov bases, we provide a global version of the independent set theorem of Hillar and Sullivant \cite[Theorem 4.7]{HS12}.
	
	\subsection{Equivariant Noetherianity} 
	
	We say that $\ZZ^{(I)}$ is an \emph{equivariantly Noetherian $\ZZ$-module} if every $\Sym$-invariant lattice in $\ZZ^{(I)}$ has a finite equivariant generating set. Our goal here is to prove the following result and discuss its consequences.
	
	\begin{theorem}
		\label{Finiteness-generating}
		$\ZZ^{(I)}$ is an equivariantly Noetherian $\ZZ$-module.
	\end{theorem}
	
	When $c=1$, this result follows from \cite[Theorem 19]{HM}. It is therefore enough to prove the following lemma.
	
	\begin{lemma}
	\label{direct-sum}
	Let $I^{(1)}=\NN^d\times[c_1]$ and $I^{(2)}=\NN^d\times[c_2]$, where $c_1,c_2\in\ZZ_{\ge0}$ with $c_1+c_2=c$. Then  $\ZZ^{(I)}$ is an equivariantly Noetherian $\ZZ$-module if $\ZZ^{(I^{(1)})}$ and $\ZZ^{(I^{(2)})}$ are.
	\end{lemma}
	
	\begin{proof}
	The natural isomorphism $\ZZ^{(I)}\cong\ZZ^{(I^{(1)})}\oplus \ZZ^{(I^{(2)})}$ yields the short exact sequence
	\[
	\xymatrix{
		0 \ar[r] & \ZZ^{(I^{(1)})} \ar[r]^-{\iota} & \ZZ^{(I)}  \ar[r]^-{\eta} & \ZZ^{(I^{(2)})}\ar[r] & 0,
	}
	\]
	where $\iota$ and $\eta$ are compatible with the $\Sym$-action. For any $\Sym$-invariant lattice $L\subseteq\ZZ^{(I)}$, it is then routine to construct a finite equivariant generating set of $L$ from finite equivariant generating sets of the lattices $\eta(L)\subseteq\ZZ^{(I^{(2)})}$ and $\iota^{-1}(L)\subseteq\ZZ^{(I^{(1)})}$.
	\end{proof}
	
	Let us now derive a few consequences of \Cref{Finiteness-generating}. First, combining \Cref{stabilization,Finiteness-generating} yields the next corollary, which can be seen as a local version of \Cref{Finiteness-generating}. 
	
	\begin{corollary}
		\label{local-Noetherianity}
		For any $\Sym$-invariant chain of lattices $\L=(L_n)_{n\ge1}$ in $\ZZ^{(I)}$, the statements (i), (ii), (iii) in \Cref{stabilization} always hold true.
	\end{corollary}
	
    The following algebraic version of \Cref{Finiteness-generating} is obtained by using \Cref{bases_algebraic}.
	
	\begin{corollary}
		\label{algebraic-Noetherianity}
		For any $\Sym$-invariant lattice $L\subseteq\ZZ^{(I)}$, the Laurent ideal $\Ik_L^\pm$ is equivariantly finitely generated.
	\end{corollary}
	
	Now \Cref{local-Noetherianity} together with \Cref{rm-ideal-stabilization} yields the following generalization of \cite[Theorem 3]{HM}.
	\begin{corollary}
		\label{local-algebraic--Noetherianity}
		For any $\Sym$-invariant chain of lattices $\L=(L_n)_{n\ge1}$ in $\ZZ^{(I)}$, the chain of Laurent ideals $(\Ik_{L_n}^\pm)_{n\ge1}$ stabilizes.
	\end{corollary}
	
	It should be noted that \Cref{Finiteness-generating} is no longer true when $d\ge2$ if generating set is replaced with Markov basis.
	
	\begin{example}
	    \label{no-3-way}
	    Consider the \emph{no $3$-way interaction model}, i.e. the hierarchical model defined by $\Delta=\{\{1,2\},\{1,3\},\{2,3\}\}$ and $\rb=(r_1,r_2,r_3)\in\NN^3$. For such a model, the $\Delta$-marginal map
	    \[
	   \mu_{\Delta,\rb}:\ZZ^{([r_1]\times[r_2]\times[r_3])}\to \ZZ^{(([r_1]\times[r_2])}\oplus\ZZ^{(([r_1]\times[r_3])}
	   \oplus \ZZ^{(([r_2]\times[r_3])}
	   \]
	   computes all $2$-way marginals of each $3$-way table in $\ZZ^{([r_1]\times[r_2]\times[r_3])}$. It is known that (see \cite{DS} and also \cite[Example 4.3]{HS12}) when $r_3=c\ge2$ is fixed and $r_1=r_2=n$ vary, any Markov basis of the $\Sym$-invariant lattice $L_n=\ker\mu_{\Delta,\rb}\subseteq\ZZ^{([n]^2\times[c])}$ contains the following element
	   \[
	   \ub=\sum_{i=1}^n(\eb_{i,i,1}-\eb_{i,i,2})
	   +
	   \sum_{i=1}^{n-1}(\eb_{i,i+1,2}-\eb_{i,i+1,1})
	   +(\eb_{n,1,2}-\eb_{n,1,1}),
	   \]
	   where $\eb_{i,j,k}$ denotes a standard basis element of $\ZZ^{(\NN^2\times[c])}$. Obviously, the support size of $\ub$ is unbounded when $n$ tends to $\infty$. So by \cref{stabilization-others}, the lattice $L=\bigcup_{n\ge1}L_n\subseteq\ZZ^{(\NN^2\times[c])}$ has no finite equivariant Markov bases.
	   
	   When $\rb=(r_1,r_2,3)$, where $r_1,r_2$ vary, De Loera and Onn \cite[Theorem 1.2]{DO} even showed that Markov bases of $\ker\mu_{\Delta,\rb}$ can be arbitrarily complicated: for any $\vb\in\NN^k$, there exist $r_1,r_2$ such that any Markov basis of $\ker\mu_{\Delta,\rb}$ contains an elements whose restriction to some $k$ entries is precisely $\vb$.
	\end{example}
	
	\subsection{The independent set theorem}
	
	In view of \Cref{no-3-way}, it is of great interest to know which $\Sym$-invariant lattices in $\ZZ^{(I)}$ have finite equivariant Markov bases when $d\ge2$. We discuss here one of the major results in this direction, the independent set theorem of Hillar and Sullivant \cite[Theorem 4.7]{HS12}, which is a generalization of the finiteness result of Ho\c{s}ten and Sullivant \cite{HoS} mentioned in the previous section. For related results see \cite{Do,DoS,DEKL,KKL}.
	
	Consider the hierarchical model determined by a simplicial complex $\Delta\subseteq 2^{[m]}$ and a vector $\rb=(r_1,\dots,r_m)\in\NN^m$. In the independent set theorem, several entries of $\rb$, indexed by an independent set, can vary, while the others are fixed. Here, a subset $T\subseteq[m]$ is called an \emph{independent set} of $\Delta$ if $|T\cap F_k|\le 1$ for all $k\in [s]$, where $F_1,\dots,F_k$ are the facets of $\Delta$. Fix $r_j$ for $j\in[m]\setminus T$ and let $r_i=n$ for all $i\in T$. Denote $d=|T|$, $c=\prod_{j\in[m]\setminus T} r_j$ and $d_k=|T\cap F_k|$, $c_k=\prod_{j\in F_k\setminus T} r_j$ for $k\in [s]$. Then $\R=\prod_{i=1}^m[r_i]$ and $\R_{F_k}=\prod_{j\in F_k}[r_j]$ can be identified with $[n]^d\times[c]$ and $[n]^{d_k}\times[c_k]$, respectively. The $\Delta$-marginal map now becomes a map
	\[
	\mu_{\Delta,\rb}:\ZZ^{([n]^d\times[c])}\to \bigoplus_{k=1}^s\ZZ^{([n]^{d_k}\times[c_k])}.
	\]
	We will write $\mu_{\Delta,\rb}$ as $\mu_{\Delta,T,n}$ to emphasize its dependence on $T$ and $n$. Consider the action of $\Sym(n)$ on $\ZZ^{([n]^d\times[c])}$ and $\ZZ^{([n]^{d_k}\times[c_k])}$ as before. It is easy to see that $\mu_{\Delta,T,n}$ is compatible with the $\Sym(n)$-action. Thus, the lattice $\ker\mu_{\Delta,T,n}$ is $\Sym(n)$-invariant, and moreover, the chain $\L_{\Delta,T}=(\ker\mu_{\Delta,T,n})_{n\ge1}$ is $\Sym$-invariant.
	
	With the above setup, the independent set theorem of Hillar and Sullivant can be stated as follows.
	
	\begin{theorem}
	    \label{independent-set}
	    Consider the hierarchical model defined by $\Delta\subseteq 2^{[m]}$ and  $\rb=(r_1,\dots,r_m)$. Suppose $T\subseteq[m]$ is an independent set of $\Delta$. Fix $r_j$ for $j\in[m]\setminus T$ and let $r_i=n$ vary for all $i\in T$. Then the $\Sym$-invariant chain of lattices $\L_{\Delta,T}=(\ker\mu_{\Delta,T,n})_{n\ge1}$ Markov-stabilizes.
	\end{theorem}
	
	Extending the maps $\mu_{\Delta,T,n}$, we obtain the \emph{global $\Delta$-marginal map}
	\[
	\mu_{\Delta,T}:\ZZ^{(\NN^d\times[c])}\to \bigoplus_{k=1}^s\ZZ^{(\NN^{d_k}\times[c_k])}
	\]
	with $\mu_{\Delta,T}\restr{\ZZ^{([n]^d\times[c])}}=\mu_{\Delta,T,n}$ for all $n\ge 1$. Evidently, $\ker\mu_{\Delta,T}=\bigcup_{n\ge1}\ker\mu_{\Delta,T,n}$. \Cref{stabilization-others} thus yields the following global version of \Cref{independent-set}.
	
	\begin{corollary}
	    \label{global-independent-set}
	    Keep the assumption of \Cref{independent-set}. Then the lattice $\ker\mu_{\Delta,T}$ has a finite equivariant Markov basis.
	\end{corollary}
	
	Note that the assumption that $T$ is an independent set  cannot be omitted in \Cref{independent-set,global-independent-set}, as \Cref{no-3-way} shows.
	
	\begin{remark}
	     Hillar and Sullivant actually proved the algebraic version of \Cref{independent-set}, i.e. the chain of lattice ideals $(\Ik_{L_n})_{n\ge1}$ stabilizes, where $L_n=\ker\mu_{\Delta,T,n}$ for $n\ge1$. It would be interesting to have \emph{purely combinatorial} proofs of \Cref{independent-set,global-independent-set}. The framework developed in this paper might be helpful in finding such proofs.
	 \end{remark}

	     It is still open whether the lattice $\ker\mu_{\Delta,T}$ always has a finite equivariant Gr\"{o}bner basis. On the other hand, this lattice does not have a finite equivariant Graver basis in general. 
	     
	 \begin{example}
	   Consider the \emph{independence model} for 2-way tables, i.e. the hierarchical model determined by $\Delta=\{\{1\},\{2\}\}$ and  $\rb=(r_1,r_2)\in\NN^2$. In this model, $T=\{1,2\}$ is an independent set of $\Delta$. Let $r_1=r_2=n$ vary. Then it is known that the Graver basis of $\ker\mu_{\Delta,T,n}$ consists of the elements
	    \[
	    \ub=\eb_{i_1,j_1}-\eb_{i_1,j_2}+\eb_{i_2,j_2}-\eb_{i_2,j_3}
	    +\cdots+ \eb_{i_k,j_k}-\eb_{i_k,j_1},
	    \]
	    where $2\le k\le n$ and each index vector $(i_1,\dots,i_k),(j_1,\dots,j_k)\in[n]^k$ has pairwise distinct entries; see \cite[Proposition 4.2]{AHT}. Hence, $\ker\mu_{\Delta,T}$ has no finite equivariant Graver bases by \Cref{stabilization-others}.
	 \end{example}
	 
	 We close this section with the following cofiniteness property of the lattice $\ker\mu_{\Delta,T}$.
	 
	 \begin{proposition}
	     \label{cofiniteness}
	    Keep the assumption of \Cref{independent-set}. Then the quotient abelian group $\ZZ^{(\NN^d\times[c])}/\ker\mu_{\Delta,T}$ is a lattice that has a finite equivariant Graver basis.
	 \end{proposition}
	 
	 The proof is based on the next result that slightly generalizes \Cref{Finiteness-Graver}.
	 
	 \begin{lemma}
	   \label{Finiteness-Graver-ext}  
	   Let $c,m\in\ZZ_{\ge0}$. Consider the action of $\Sym$ on $\ZZ^{(\NN\times[c])}\oplus\ZZ^m$ that extends its action on $\ZZ^{(\NN\times[c])}$ with a trivial action on $\ZZ^m$, i.e.
	   \[
	   \sigma(\ub,\vb)=(\sigma(\ub),\vb)
	   \ \text{ for any } \ub\in \ZZ^{(\NN\times[c])}, \vb\in \ZZ^m, \sigma\in\Sym.
	   \]
	   Then every $\Sym$-invariant lattice in $\ZZ^{(\NN\times[c])}\oplus\ZZ^m$ has a finite equivariant Graver basis.
	 \end{lemma}
	 
	 \begin{proof}
	     Let $L\subseteq \ZZ^{(\NN\times[c])}\oplus\ZZ^m$ be an arbitrary $\Sym$-invariant lattice. Arguing similarly to the proofs of \Cref{prop_Hilbert_Graver,cor_Hilbert_Graver}, one can show that $L$ has a finite equivariant Graver basis if and only if the monoid $\varphi^{-1}(L)\cap (\ZZ_{\ge0}^{(\NN\times[2c])}\oplus\ZZ_{\ge0}^{[2m]})$ has a finite equivariant Hilbert basis, where
	     \[
	     \varphi: \ZZ^{(\NN\times[2c])}\oplus\ZZ^{[2m]}\cong
	     (\ZZ^{(\NN\times[c])}\oplus\ZZ^m)\oplus(\ZZ^{(\NN\times[c])}\oplus\ZZ^m)
	     \to \ZZ^{(\NN\times[c])}\oplus\ZZ^m
	     \]
	     is defined as in \Cref{sec-Graver}. So by renaming $c$ and $m$, it suffices to show that the monoid $M=L'\cap (\ZZ_{\ge0}^{(\NN\times[c])}\oplus\ZZ_{\ge0}^{[m]})$ has a finite equivariant Hilbert basis for any $\Sym$-invariant lattice $L'\subseteq \ZZ^{(\NN\times[c])}\oplus\ZZ^m$. Consider the following partial order on $\ZZ_{\ge0}^{(\NN\times[c])}\oplus\ZZ_{\ge0}^{[m]}$:
	     \[
	     (\ub,\vb)\preceq (\ub',\vb')
	     \ \text{ if }\ \ub\sqsubseteq_{\Inc} \ub'
	     \ \text{ and } \vb\sqsubseteq \vb',
	     \]
	     where $\sqsubseteq_{\Inc}$ is defined in the proof of \Cref{lattice-monoid-ext}. Since $\sqsubseteq_{\Inc}$ and $\sqsubseteq$ are both well-partial-orders, it is easy to check that $\preceq$ is also a well-partial-order (see, e.g. \cite{Dr14}). Now an argument similar to the proof of \Cref{lattice-monoid-ext} shows that $M$ has a finite equivariant Hilbert basis, as desired.
	 \end{proof}
	     
	Let us now prove \Cref{cofiniteness}.
	
	\begin{proof}[Proof of \Cref{cofiniteness}]
	    First, $\ZZ^{(\NN^d\times[c])}/\ker\mu_{\Delta,T}$ is a lattice because it is isomorphic to a subgroup of $\bigoplus_{k=1}^s\ZZ^{(\NN^{d_k}\times[c_k])}$. Since $T$ is an independent set of $\Delta$, either $d_k=0$ or $d_k=1$ for every $k\in[s]$. It follows that $\bigoplus_{k=1}^s\ZZ^{(\NN^{d_k}\times[c_k])}$ is isomorphic to $\ZZ^{(\NN\times[c])}\oplus\ZZ^m$ for some $c,m\in\ZZ_{\ge0}$. The desired conclusion is now immediate from \Cref{Finiteness-Graver-ext}.
	\end{proof}


\end{document}